\documentclass[11pt]{article}
\usepackage[latin1]{inputenc}
\usepackage{amsmath}
\usepackage{amsfonts}
\usepackage{amsthm}
\usepackage{amssymb,amsthm, stmaryrd, mathabx}
\usepackage{graphicx}
\usepackage{xcolor}
\usepackage{multicol}
\usepackage{soul}
\usepackage[normalem]{ulem}
\usepackage{float}

\newtheorem{theorem}{Theorem}[section]
\newtheorem{conj}[theorem]{Conjecture}
\newtheorem{proposition}[theorem]{Proposition}
\newtheorem{problem}{Problem}
\newtheorem{question}{Question}
\newtheorem{corollary}[theorem]{Corollary}
\newtheorem{defi}[theorem]{Definition}
\newtheorem{lemma}[theorem]{Lemma}
\newtheorem{remark}[theorem]{Remark}
\newtheorem{example}[theorem]{Example}
\newtheorem{obs}[theorem]{Observation}

\newcommand{\fer}{{\rm Fer}}

\def\restrict#1{\raise-.5ex\hbox{\ensuremath|}_{#1}}

\title{Graphs isomorphisms under edge-replacements\\ and the family of amoebas}

\begin{document}

\maketitle

\begin{center}

\begin{multicols}{2}

Yair Caro\\[1ex]
{\small Dept. of Mathematics\\
University of Haifa-Oranim\\
Tivon 36006, Israel\\
yacaro@kvgeva.org.il}

\columnbreak

Adriana Hansberg\\[1ex]
{\small Instituto de Matem\'aticas\\
UNAM Juriquilla\\
Quer\'etaro, Mexico\\
ahansberg@im.unam.mx}\\[2ex]

\end{multicols}

Amanda Montejano\\[1ex]
{\small UMDI, Facultad de Ciencias\\
UNAM Juriquilla\\
Quer\'etaro, Mexico\\
amandamontejano@ciencias.unam.mx}\\[4ex]

\end{center}

\begin{abstract}
This paper offers a systematic study of a family of graphs called amoebas. Amoebas recently emerged from the study of forced patterns in $2$-colorings of the edges of the complete graph in the context of Ramsey-Turan theory and played an important role in  extremal  zero-sum problems. 
Amoebas are graphs with a unique behavior with regards to the following operation: Let $G$ be a graph and let $e\in E(G)$ and $e'\in E(\overline{G})$.  If the graph $G'=G-e+e'$ is isomorphic to $G$, we say $G'$ is obtained from $G$ by performing a \emph{feasible edge-replacement}.  We call $G$ a \emph{local amoeba} if, for any two copies $G_1$, $G_2$ of $G$ on the same vertex set, $G_1$ can be transformed into $G_2$ by a chain of feasible edge-replacements. On the other hand, $G$ is called \emph{global amoeba} if there is an integer $t_0 \ge 0$ such that $G \cup tK_1$ is a local amoeba for all $t \ge t_0$. 
To model the dynamics of the feasible edge-replacements of $G$, we define a group $\fer(G)$ that satisfies that $G$ is a local amoeba if and only if $\fer(G) \cong S_n$, where $n$ is the order of $G$. Via this algebraic setting, a deeper understanding of the structure of amoebas and their intrinsic properties comes into light.  
Moreover, we present different constructions that prove the richness of these graph families showing, among other things, that any connected graph can be a connected component of a global amoeba, that global amoebas can be very dense and that they can have, in proportion to their order, large clique and chromatic numbers. Also, a family of global amoeba trees with a Fibonacci-like structure and with arbitrary large maximum degree is constructed.
\end{abstract}

\section{Introduction}

Graphs called amoebas first appeared in \cite{CHM3} where certain Ramsey-Tur\'an extremal problems were considered, which dealt with the existence of a given graph with a prescribed color pattern in $2$-edge-colorings of the complete graph. More precisely, amoebas arose from the search of a graph family with certain interpolation properties that could be suitable to show balanceability or omnitonal properties \cite{CHM3} (see also \cite{CHLZ}). The graphs named just ``amoebas'' in \cite{CHM3} are called in this paper, with their rich, Matroid resembling properties, ``global amoebas'', as we will distinguish them from a similar family that we call ``local amoebas''. For the interested reader, we refer to \cite{BaNiWh, HaMoPl, HaPl, LiLiTo, Pun1, Pun3, Zhou2} for more literature related to interpolation techniques in graphs. 

The feature that makes amoebas work are one-by-one replacements of edges, where, at each step, some edge is substituted by another such that an isomorphic copy of the graph is created. We call such edge substitutions feasible edge-replacements.  Similar edge-operations have been studied, for instance, in \cite{CiZi, FrRi, FrHlRo, JaPaSc, Pun3, Ross}.  As introduced in \cite{CHLZ}, a family $\mathcal{F}$ of graphs, all of them having the same number of edges, is called \emph{closed} in a graph $H$ if, for every two copies $F, F'$ of members of $\mathcal{F}$ contained in $H$, there is a chain of graphs $H_1, H_2, \ldots, H_k$ in $H$ such that $H_1 \cong F$,  $H_k \cong F'$, and, for $2 \le i \le  k$,  $H_i$ is isomorphic to a member of $\mathcal{F}$ and $H_i$ is obtained from $H_{i-1}$ by interchanging one edge with another (an edge-replacement that is not necessarily feasible).  Perhaps the most well-known closed family is the family of all spanning trees of a connected graph $H$, and the edge-replacement operation given above is in fact the basic operation in the exchange of bases in the cycle matroid $M(H)$ of $H$. A graph $G$ is a \emph{global amoeba} precisely when $\{G\}$ is a closed family in $K_n$ (for $n$ large enough), and it is a \emph{local amoeba} if $\{G\}$ is a closed family in $K_{n(G)}$. Exactly this global amoeba property is the key to the usefulness of amoebas in interpolation theorems in Graph Theory  and in zero-sum extremal problems \cite{CHLZ}, and in problems about forced patterns in $2$-colorings of the edges of $K_n$ \cite{CHM3}. We note at this point, once again, that the amoebas defined in \cite{CHM3} correspond to the class of global amoebas.

A first encounter with amoebas gives the impression that such graphs are very rare and have a very simple structure. This, however, is not the case and amoebas may have quite a complicated structure. Indeed, we will consider here different constructions with which we will show that any connected graph can be a connected component of a global amoeba (Theorem \ref{thm:conn_comp}), that global amoebas can be very dense  (in fact, with as many as $n^2/4$ edges, being $n$ the order of the graph) and, that they can have very large chromatic number and cliques, too  (as large as roughly half the order of the graph) (Theorem \ref{thm:GA_max_e(G)X(G)w(G)}). Also, we introduce an interesting family of global amoeba trees with a Fibonacci-like structure and with arbitrary large maximum degree (Theorem \ref{thm:Ti_GA}). 

Most concepts and definitions concerning amoebas can be stated in graph theoretical language. However, a group theoretical setting with which the dynamics of the edge replacements are modeled -- and which involve graph isomorphisms -- will be necessary for proving several results. In particular, the proof of Theorem 3.8, which is a key element for many other results, employs tools of permutation groups, which via graph theoretical language would be too intricate. This is the reason why we will develop an algebraic theoretical setting so that we can then formalize all concepts and definitions by means of the permutation language, and we will proceed further on with the theory this way. The research on amoebas is related to problems like switching in graphs \cite{BaMoRa, Hor}, reconfiguration problems \cite{IDHPSUU, MeNoPo}, token graphs \cite{CaDuPa, FFHHUW, LeTr} and, with respect to aspects of group action language, to reconstruction problems in graphs \cite{Czi05, HKNS10, LaSc}. Similar approaches that deal with graph isomorphisms can be found in \cite{ACKR89, RaSc06, Sep19}. For group theoretical concepts, we refer to \cite{Isa}.

The paper is organized as follows. In Section \ref{sec:theor_set1}, we define local and global amoebas formally by means of graph theoretical language, and present some basic results achievable by graph theoretical tools. In Section \ref{sec:theor_set2}, we introduce the group theoretical tools that will let us model how the so-called feasible edge-replacements work via permutations. By means of this algebraic setting, we will reformulate the definitions of global amoeba and local amoeba (Definition \ref{def:amoebas}). Section \ref{sec:main} contains our main result (Theorem \ref{thm:eq}) which displays different characterizations of global amoebas by which the relation between global and local amoebas is very clearly established. Moreover, it is shown, using non-trivial examples, how the characterization is very handy.  In Section \ref{sec:constr},  we will present some interesting constructions of both local and global amoebas that will exhibit the richness of this family of graphs. \\

\begin{figure}[H]
\begin{center}
	\includegraphics[scale=0.48]{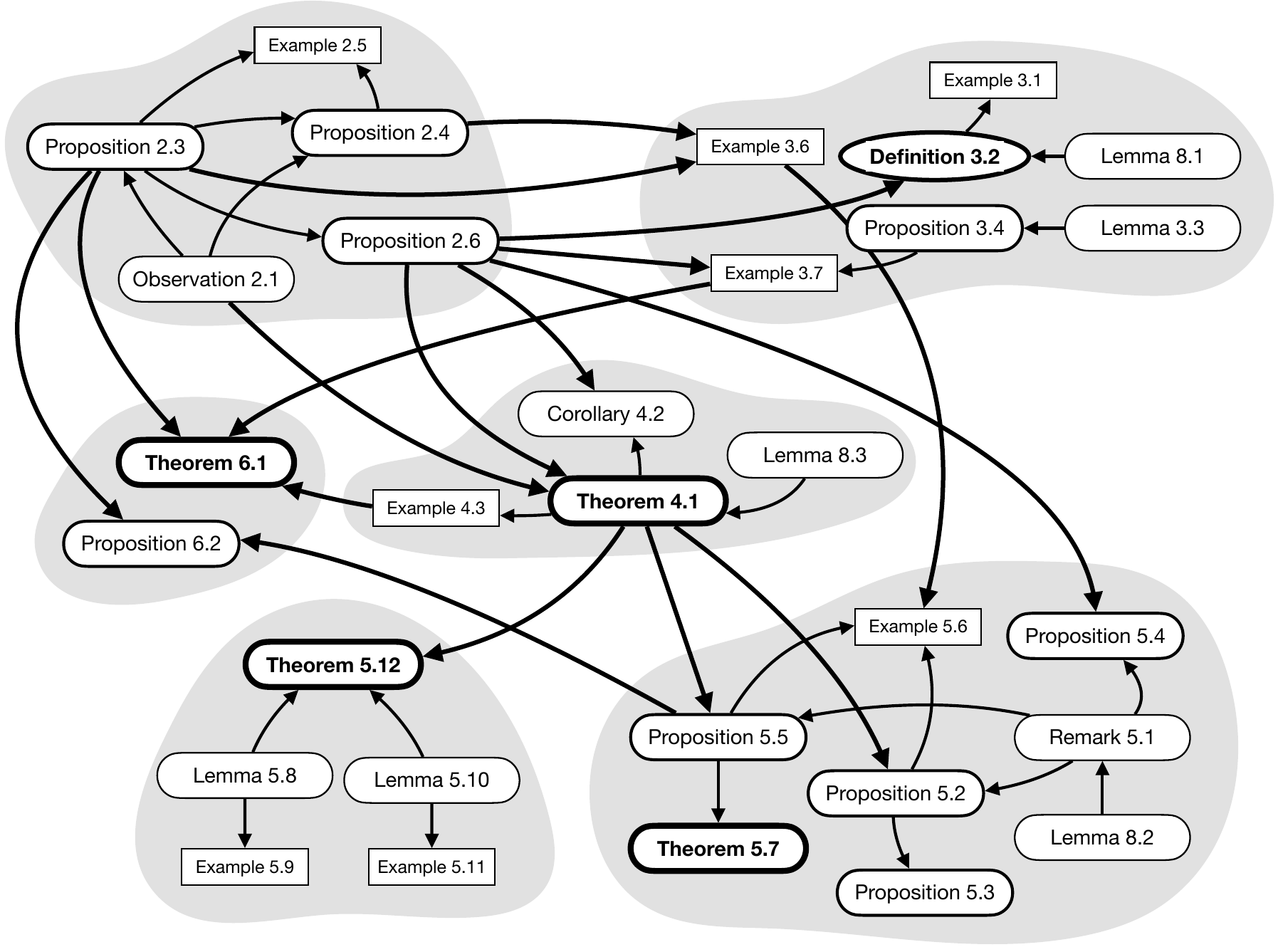}
	\caption{Results dependence scheme}
	\label{fig:structure}
\end{center}	
\end{figure}

For many of these results, Theorem \ref{thm:eq} is a crucial tool. In Section \ref{sec:e(G)X(G)w(G)}, we exhibit extremal global amoebas with respect to size, chromatic number and clique number. In order to show the purpose of the results, we will illustrate with abundant examples. In the final section, we provide the reader with several open problems which could bring more light to understanding this very interesting family of graphs called amoebas. Figure \ref{fig:structure} depicts a scheme that captures the structure and the dependence between the results of this paper. The scheme includes Definition \ref{def:amoebas} to emphasize that the developed theory in Section \ref{sec:theor_set1} is needed to be able to establish the definition of local and global amoebas via the algebraic approach. It is important to note that all subsequent results depend on this definition and its group theoretical setting.

\section{Amoebas: graph theoretical approach}\label{sec:theor_set1}

As it was said in the introduction, the graph class of amoebas arose from the search of a family with nice interpolation properties that work in solving  certain Ramsey-Tur\'an extremal problems in $2$-edge-colorings of the complete graph, see \cite{CHM3}. In order to define amoebas, we need to formally establish what a feasible edge-replacement is.

Let $G$ be a nonempty and non-complete graph. Given $e\in E(G)$ and $e'\in E(\overline{G})
$, we say that the graph $G' = G-e+e'$ is obtained from $G$ by replacing the edge $e$ with $e'$. 
If $G'$ is a graph isomorphic to $G$, we say that the edge-replacement is \emph{feasible} and that \emph{$G'$ is obtained from $G$ by a feasible edge-replacement}. We also need to consider the \emph{neutral edge-replacement} 
as a feasible edge-replacement,  which is given when no edge is replaced at all.  
Note that every graph has at least one feasible edge replacement, namely the neutral edge-replacement.

The following observation is a direct consequence of the definition of feasible edge-replacement.

\begin{obs}\label{obs:trivialDEG}
Let $G$ be a nonempty and non-complete graph. Let $uv \in E(G)$ and $u'v' \in E(\overline{G})$. Suppose that $G' = G -uv + u'v' \cong G$. Then we have, for any vertex $w \in V(G)$,
\begin{align}
\deg_{G'}(w) = \left\{ \begin{array}{ll}
\deg_G(w)-1, \mbox{ if } w \in \{u,v\}  \setminus \{u',v'\}\\
\deg_G(w)+1, \mbox{ if } w \in \{u',v'\}  \setminus \{u,v\}\\
\deg_G(w), \mbox{ else.}\\
\end{array}
\right.
\end{align}  
\end{obs}

Thus, for a graph that is nonempty and non-complete, there are two types of non-neutral feasible edge-replacements (if any), according to how many vertices ($0$ or $1$) share the edge we remove and the edge we insert. Now we will define global and local amoebas.\\

\begin{defi}
A graph $G$ is called a \emph{local amoeba} if, for any two isomorphic copies of $G$ on the same vertex set $V = V(G)$, say $F$ and $H$, there is a chain $F=G_0,G_1,G_2,...,G_{\ell} =H$ such that, for every $1\leq i \leq \ell$, $G_i\cong G$ and $G_i$ is obtained from $G_{i-1}$ by a feasible edge-replacement. Moreover, $G$ is called \emph{global amoeba} if there is an integer $t_0 \ge 0$ such that $G \cup tK_1$ is a local amoeba for every $t \ge t_0$.
\end{defi}

By means of this innocent technical definition of global amoeba, one can imagine how it is possible to move from any copy of a global amoeba $G$ in $K_n$ to any other copy of $G$ in $K_n$, if $n$ is large enough. We shall see later on that  $n =  n(G) +1$ (i.e. $t_0 = 1$) suffices. The following is a simple but very useful result.

\begin{proposition}\label{prop:la_degrees}
Let $G$ be a graph with minimum degree $\delta$ and maximum degree $\Delta$.  If $G$ is a local amoeba then, for every integer $r$ with $\delta \leq r \leq \Delta$, there is a vertex $v\in V(G)$ with $\deg_G(v)=r$. If $G$ is a global amoeba the same is true and, moreover,  $\delta\leq 1$. 
\end{proposition}

\begin{proof}
Let $G$ be a local amoeba embedded in $K_n$, and let $v\in V(G)$ with $\deg_G(v)=\delta$. Let $H$ be an isomorphic copy of $G$ embedded in $K_n$ with $\deg_H(v)=\Delta$. Since $G$ is a local amoeba we know there is a chain $G=G_0,G_1,G_2,...,G_{\ell} =H$ such that, for every $1\leq i \leq \ell$, $G_i\cong G$ and $G_i$ is obtained from $G_{i-1}$ by a feasible edge-replacement. Then, by setting  $d_i=\deg_{G_i}(v)$, we have a sequence of integers $d_0,d_1,...,d_{\ell}$ where $d_0=\delta$, $d_{\ell}=\Delta$ and, by Observation \ref{obs:trivialDEG}, for every $0\leq i\leq\ell -1$, we have $d_i-1\leq d_{i+1}\leq d_i+1$. Now, fix $r$ with $\delta +1\leq r \leq \Delta-1$, and let $m=\max \{i:d_i\leq r\}$. We claim that $d_m=r$, otherwise, if $d_m<r$ then $d_{m+1}\leq d_m+1<r+1$, thus $d_{m+1}\leq r$ which is a contradiction (by the definition of $m$). Hence, we find a vertex of degree $r$ in $G_m$ which is graph isomorphic to $G$. 

If $G$ is a global amoeba, let $G$ be embedded in $K_N$ for some $N>n$. Let $v\in V(K_N)\setminus V(G)$ and recall that we set  $\deg_G(v) = 0$.  To complete the proof, we proceed as in the previous paragraph to find vertices of degree $r$ for every $1\leq r \leq \Delta-1$.
\end{proof}

By Observation \ref{obs:trivialDEG}, regular graphs have the neutral edge-replacement as its only feasible edge-replacement. Hence, no regular graph, except  for the complete or the empty graph, can be a local amoeba. Similarly, no regular graph, except for $\frac{n}{2}K_2$, for even $n$, or $\overline{K_n}$ can be a global amoeba. We use Proposition \ref{prop:la_degrees} to formalize the above concerning not only regular graphs but also graphs that have the neutral edge-replacement as its only feasible edge-replacement. 

\begin{proposition}\label{prop:neutralFER}
Let $G$ be a graph of order $n\geq 1$ having the neutral edge-replacement as its only feasible edge-replacement.  Then 
\begin{enumerate}
    \item[(i)] $G$ is  a local amoeba if and only if  either $G=K_n$, or  $G=\overline{K_n}$;
    \item[(ii)] $G$ is a global amoeba if and only if either $G= \frac{n}{2}K_2$, for even $n$, or $G=\overline{K_n}$.
\end{enumerate}
\end{proposition}

\begin{proof}
(i) Supposing that $G$ is a local amoeba, the fact of $G$ having the neutral edge-replacement as its only feasible edge-replacement means that every isomorphic copy $G'$ of $G$ on the same vertex set $V(G)$ has the same set of edges as $G$, that is, $E(G') = E(G)$. This can only happen when $G=K_n$, or $G=\overline{K_n}$, which are clearly local amoebas.\\

\noindent
(ii) The fact that $G=\overline{K_n}$ is a global amoeba having only the neutral edge-replacement as a feasible edge-replacement is obvious. Since $G= \frac{n}{2}K_2$ is a $1$-regular graph, then its only feasible edge-replacement is the neutral one. To see that   $G= \frac{n}{2}K_2$ is a global amoeba is not hard if we consider enough extra isolated vertices. We leave it as an exercise to argue that  $\frac{n}{2}K_2\cup tK_1$ is a local amoeba for every $t\geq n+4$ (in any case, we will provide a general argument to prove this in Section \ref{sec:theor_set2}).

Suppose now that $G$ is a global amoeba having the neutral edge-replacements as its only feasible edge-replacement.  By Proposition \ref{prop:la_degrees}, $G$ has minimum degree $\delta(G) \le 1$. If $\delta (G)= 0$, then $G=\overline{K_n}$ (otherwise, $G$ would have vertices of positive degree and, by Proposition \ref{prop:la_degrees}, there would be a vertex $x$ of degree $1$, which would imply the existence of a feasible edge replacement in $G$ involving the edge incident to $x$ and the isolated vertex). If $\delta(G) = 1$, we will prove that $G= \frac{n}{2}K_2$ by showing that $\Delta(G)=1$. Suppose to the contrary that  $\Delta(G) \ge 2$. By definition, there is an integer $t_0 \ge 0$ such that $G \cup tK_1$ is a local amoeba for every $t \ge t_0$. Since $G$ has a vertex of degree at least $2$, there must be two copies $H$ and $H'$ of $G \cup tK_1$ such that $H'$ is obtained from $H$ by a feasible edge-replacement and such that a vertex $v$ of $H$ with $\deg_H(v) \ge 2$ gets its degree reduced by one. By Observation~\ref{obs:trivialDEG}, this can only happen if there are vertices $u, u', v'$ such that $H' = H - uv + u'v'$ and $v\notin \{u',v'\}$.  Now, the fact that $G$ has no feasible edge-replacement different from the neutral one implies that all feasible edge-replacements of $H$ have to involve isolated vertices. Thus, some of $u'$ or $v'$ is an isolated vertex, say $\deg_H(u')=0$. This forces that $\deg_H(u)=1$ and, since $H$ and $H'$ are isomorphic graphs, it follows that $\deg_{H}(v') =\deg_H(v)$. But, in such a situation, replacing the edge $uv$ by $uv'$ would represent a feasible non-neutral edge-replacement of $G$, which does not exist. Hence, $\Delta(G) = 1$ and we are done.
\end{proof}

In the following example, we exhibit simple  graphs concerning all possibilities of being, or not, a local or a global amoeba. Items (i), (ii), and (iii) follow easily or directly from Propositions~\ref{prop:la_degrees} and \ref{prop:neutralFER}. We leave the formal proof of item (iv) as an exercise for the reader. Later on, we will provide other arguments by which this can be proven quite smoothly (see Example \ref{ex:PnandtK2}).

\begin{example}\label{ex:amoebas_basic}
\mbox{}
\begin{enumerate}
\item[(i)]  Stars $K_{1,n-1}$, with $n\geq 4$, are neither local nor  global amoebas. Also, graphs of order $n$ that are $r$-regular, for $2 \le r \le n-2$, are neither local nor  global amoebas. 

\item[(ii)]  For every $n\geq 3$, the complete graph $K_n$ is a local but not a global amoeba. 

\item[(iii)]  For every even $n\geq 4$, the graph $\frac{n}{2}K_2$ is a global but not a local amoeba.  

\item[(iv)]  For every $n\geq 2$, the  path $P_n$ on $n$ vertices is  both, a local amoeba and a global amoeba. 
\end{enumerate}
\end{example}

With the aim of showing the richness of both classes of local and global amoebas,  we will see more examples throughout the paper.  The following proposition gives us useful information about graphs with minimum degree $0$ or $1$.

\begin{proposition}\label{prop:localUK1=local} 
Let $G$ be a graph of order $n$. If $G$ is a local amoeba with $\delta(G) \in \{0, 1\}$, then $G \cup K_1$ is a local amoeba, and so $G$ is a global amoeba.
\end{proposition}

\begin{proof}
Let $G$ be a local amoeba of order $n$ with $\delta(G) \in \{0, 1\}$. First of all note that,  if $G = \overline{K_n}$, we are done. Hence, $\delta(G) \le 1 \le \Delta(G)$, which, in view of Proposition \ref{prop:la_degrees}, implies that $G$ has a vertex of degree $1$. 

Define $G'=G\cup K_1$. We need to prove that for any two isomorphic copies of $G'$ on the vertex set $V'= V(G')$, say $F'$ and $H'$, there is a chain of feasible edge-replacements that takes $F'$ to $H'$. Since both $F'$ and $H'$ are isomorphic to $G'$, we know that there are vertices $u$ and $w$ such that  $\deg_{F'}(u)=0$ and $\deg_{H'}(w)=0$. We consider two cases: either $u=w$, or $u\neq w$. In the first case, we are done, since $F'-u\cong H'-w\cong G$ and $G$ is a local amoeba. In the second case, we will make use of two auxiliary copies of $G$, $J'$ and $K'$, such that, by means of feasible edge-replacements, we will be able to transform $F'$ into $J'$, $J'$ into $K'$, and finally $K'$ into $H'$.
To this aim, consider a vertex $v$ of degree $1$ in $F'$. Note that both $v$ and $w$ belong to $V(F')\setminus \{u\}$. Since $F'-u\cong G$ and  $G$ is a local amoeba, we know there is a  chain of feasible edge-replacements that takes $F'-u$ into a graph $J$, isomorphic to $G$, where $\deg_{J}(w)=1$. But this chain of  feasible edge-replacements is also a chain of  feasible edge-replacements that takes $F'$ to a graph $J'$, where $u$ is an isolated vertex of $J'$ (that is, $J= J' - u$), and $\deg_{J'}(w)=1$. Let $y$ be the neighbor of $w$ in $J'$. Consider now the graph $K' = J' - yw +yu \cong G'$. Since $K = K'-w \cong H'-w \cong G$ and $G$ is a local amoeba, there is a chain of feasible edge-replacements that transforms $K$ into $H'-w$. Clearly, this chain of feasible edge-replacements can be also used to transform $K$ into $H$.
\end{proof}

In view of Proposition \ref{prop:localUK1=local} we conclude that, if one knows that a graph $G \cup t_0K_1$ is a local amoeba for some integer $t_0 > 0$, then $G \cup tK_1$ is a local amoeba for all $t \ge t_0$. Hence, to prove that a graph $G$ is a global amoeba, one has only to find a $t \geq 1$ for which $G \cup t K_1$ is a local amoeba. This means that the definition of global amoeba can be established as: $G$ is a global amoeba if there is an integer $t \ge 1$ such that $G \cup tK_1$ is a local amoeba.

A natural question that arises at this point is the following.
 
\begin{question}\label{q:t=1}
Let $G$ be a global amoeba. What is the minimum $t$ for which  $G \cup tK_1$ is a local amoeba? Does the value of such minimum $t$ depend on the structure of $G$?
\end{question}

Interestingly, it turns out that we just need $t =1$, and that occurs for any global amoeba. This is shown in Theorem \ref{thm:eq}, which is  a major achievement of this paper.

\section{Amoebas: algebraic approach}\label{sec:theor_set2}

For positive integers $m$ and $n$ with $m<n$ we use the standard notation $[n]=\{1,2,...,n\}$ and $[m,n]=\{m,m+1,m+2,,...,n\}$. For a finite set $X$, let $S_X$ be the symmetric group, whose elements are permutations of $X$, and let $S_n = S_{[n]}$. We will use the standard cycle notation when referring to particular permutations. The group of automorphisms of a graph $G$ is denoted by ${\rm Aut}(G)$. 
Throughout this paper, every graph $G$ we consider will be equipped with a labeling on its vertex set $\lambda : V(G) \to X$, which will be always a bijection. We define $v_x = \lambda^{-1}(x)$, for each $x \in X$. Let $L_G = \{ ij \;|\; v_iv_j \in E(G)\}$ be the set of labels of the edges of $G$, where we do not distinguish between $ij$ and $ji$.  For a permutation $\sigma \in S_{X}$, we define $G_{\sigma}$
as the copy of $G$ on the vertex set $V(G_{\sigma}) = V(G)$ and edge set 
\[E(G_{\sigma}) =  \{ v_iv_j  \;|\; \sigma(i)\sigma(j)\in L_G \}.\]

Observe that, for every graph $G'$ on $V(G)$ isomorphic to $G$ there are  $|{\rm Aut}(G)|$ different copies $G_{\sigma}$ that correspond to $G'$. That is, the set  $\{\sigma \in S_X \,|\, G_{\sigma}=G'\}$ has  $|{\rm Aut}(G)|$ elements. Moreover,  $\{\sigma \in S_X \,|\, G_{\sigma}=G\}\cong {\rm Aut}(G)$. We will set 
\[A_G = \{\sigma \in S_X \,|\, G_{\sigma}=G\}.\]

\begin{example}\label{ex:P4_0}
Let $G  \cong P_4$ with $V(G)=\{v_1,v_2,v_3,v_4\}$ and $E(G)=\{v_1v_2,v_2v_3,v_3v_4\}$. Given the labeling $\lambda : V(G) \to [4]$ with $\lambda(v_i) = i$, $i \in [4]$, we have $L_G=\{ 12,23,34\}$ and $\{\sigma \in S_4 \,|\, G_{\sigma}=G\}=\{id, (14)(23)\}\cong {\rm Aut}(P_4)$. For $G'$, the isomorphic copy of $G$  defined by $E(G')=\{v_1v_3,v_2v_3,v_2v_4\}$, we have two permutations, namely  $(23)$ and $(14)$, that satisfy $G_{(23)}=G_{(14)}=G'$. See Figure~\ref{fig:exP4_0} to visualize the corresponding labelings and observe that, in all cases,  $E(G_{\sigma}) =  \{ v_iv_j  \;|\; \sigma(i)\sigma(j) \in L_G \}$. For example, $E(G_{(23)})=  \{ v_iv_j  \;|\; \sigma(i)\sigma(j) \in L_G \}=\{v_1v_3,v_3v_2,v_2v_4\}$.
\end{example}

\begin{figure}[H]
\begin{center}
	\includegraphics[scale=0.5]{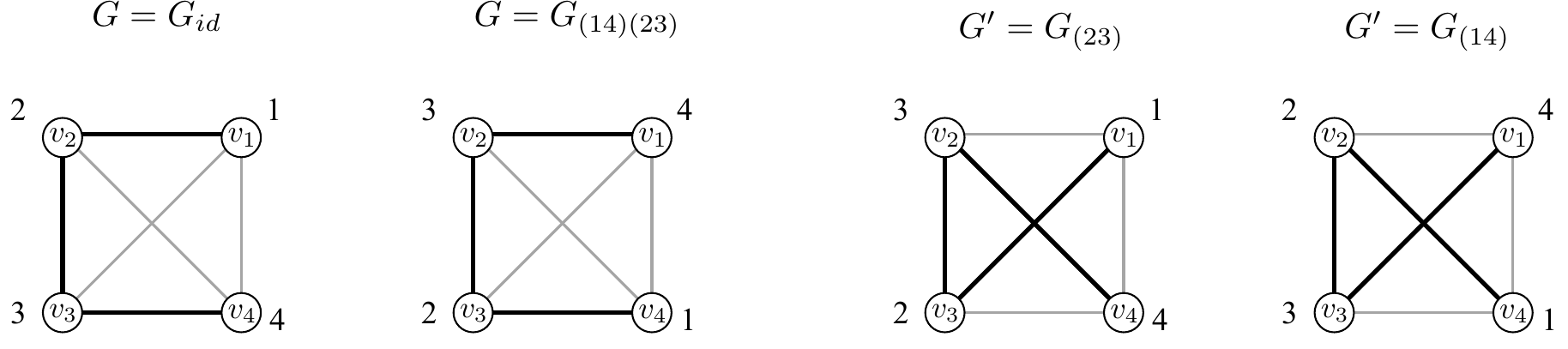}
	\caption{\small For $G=P_4$ with $V(P_4)=\{v_1,v_2,v_3,v_4\}$, $E(P_4)=\{v_1v_2,v_2v_3,v_3v_4\}$, and $\lambda(v_i) = i$, $i \in [4]$, we have $L_G=\{12,23,34\}$. The labelings corresponding to the permutations ${id},{(14)(23)},{(23)}$ and ${(14)} $ are depicted (left to right) showing the copies $G=G_{id}=G_{(14)(23)}$ and  $G'=G_{(23)}=G_{(14)}$, where $E(G_{\sigma})= \{ v_iv_j  \;|\; \sigma(i)\sigma(j) \in L_G \}$ in all cases.}
	\label{fig:exP4_0}
\end{center}	
\end{figure}

The key point of using labels on the vertices is to keep track of the role each vertex and each edge is playing in each of the copies of $G$. More precisely, note that the corresponding copies of the vertices and edges of $G$ in $G_{\sigma}$ are given by their labels: The copy of vertex $v_i$ of $G$ is the vertex of $G_{\sigma}$ having label $i$, while the copy of an edge $v_iv_j \in E(G)$ is the edge of $G_{\sigma}$ having label $ij$. It is important to note  that $L_{G_\sigma} = L_G$ for all $\sigma \in S_X$, i.e. the set of labels of the edges of $G$ remains invariant among all copies $G_{\sigma}$, $\sigma \in S_X$. Given the prescribed labeling, we can talk now about edge-replacements by means of elements in $L_G$. In this context, we will denote by $rs \to kl$ the edge-replacement that corresponds to deleting the edge with label $rs$ and adding the edge with label $kl$. When the edge-replacement $rs \to kl$ is \emph{feasible}, it means that $G - v_rv_s + v_kv_l \cong G$. We denote with $\emptyset \to \emptyset$ the \emph{neutral edge-replacement}, where no edge is replaced at all. We define now $$R_G = \{rs \to kl \;|\;  G - v_rv_s + v_kv_l \cong G, rs \neq kl\} \cup\{\emptyset \to \emptyset\}$$ as the set of all feasible edge-replacements of $G$ given by their labels together with the neutral edge-replacement.  Let $R_G^* = R_G \setminus \{\emptyset \to \emptyset\}$. We will use sometimes the notation $e \to e' \in R_G$ when we do not require to specify the labels of the vertices involved in the edge-replacement, and it includes the possibility that $e \to e'$ is the neutral edge-replacement.

Notice that, since feasible edge-replacements are defined by the labels of the edges, any $e \to e' \in R_G$ represents also a feasible edge-replacement of any copy $G_{\sigma}$, $\sigma \in S_X$. Hence, clearly $R_{G_{\sigma}} = R_G$ for any $\sigma \in S_X$. 

Given a feasible edge-replacement $rs \to kl\in R_G^*$ we will use the following notation   
\[\fer_G(rs \to kl)  = \{\sigma \in S_X \;|\; G_{\sigma} = G-v_rv_s+v_kv_l\},\]
and we will set $\fer_G(\emptyset \to \emptyset) = A_G$.

That is, $\fer_G(e \to e')$ is the set of permutations representing the $|{\rm Aut}(G)|$ different copies of $G$ that one can get to by means of performing the feasible edge-replacement $e \to e' \in R_G$.

Observe that performing a feasible edge-replacement $e \to e' \in R_G$ in a copy $G_{\rho}$ of $G$ yields the copy of $G$ given by the permutation $\sigma 	\, \rho$, where we can choose any  $\sigma \in  \fer_G(e \to e')$.  In other words, we can model the application of a series of feasible edge-replacements by considering the composition of the corresponding permutations. A formal proof of this rather intuitive fact can be found in Lemma \ref{la:properties_FER} in the Appendix. It now makes sense to consider the group $\fer(G)$ generated by the permutations associated to all feasible edge-replacements, that is, by the set 
\[\mathcal{E}_G =  \bigcup_{e \to e'\in R_G}  \fer_G(e \to e'). \]
Thus,
\[\fer(G) = \left\langle \mathcal{E}_G\right\rangle.\]

Clearly, $\fer(G)$ acts on the set $\{ G_{\sigma} \;|\; \sigma \in S_X\} $ by means of $(\rho,  G_{\sigma}) \mapsto G_{\rho \sigma}$, where $\rho \in \fer(G)$ and $\sigma \in S_X$. Observe that this action represents what happens when a series of feasible edge-replacements, represented by $\rho$, is performed on a copy $G_{\sigma}$ of $G$: the result is the graph $G_{\rho \sigma}$. Hence, the property of being able to go from any copy $G_{\sigma}$ to any other $G_{\sigma'}$ by means of a chain of feasible edge replacements means that, for any $\sigma, \sigma' \in S_X$, there is a $\rho \in \fer(G)$ such that $\sigma' = \rho \sigma$, i.e. $\fer(G) = S_X$.  Recall also that, by Proposition \ref{prop:localUK1=local}, if  $G \cup t_0K_1$ is a local amoeba for some $t_0 >0$, then actually  $G \cup tK_1$ is a local amoeba for every $t \ge t_0$. Thus, we can define now local and global amoebas by means of the group $\fer(G)$.

\begin{defi}\label{def:amoebas}
Let $G$ be a graph provided with a labeling $\lambda: V(G) \to X$ on its vertices. $G$ is called a \emph{local amoeba} if $\fer(G) = S_X$. 
On the other hand, $G$ is called \emph{global amoeba} if there is an integer $t \ge 1$ such that $G \cup tK_1$ is a local amoeba.
\end{defi}

We shall also note that  ${\rm Aut}(G)  \cong A_G = \{\sigma \in S_X \,|\, G_{\sigma}=G\} \leq \fer(G)$. In the next lemma, we discuss the connection between the feasible edge-replacements of a graph $G$ and those of its complementary graph $\overline{G}$, concluding that the corresponding associated groups are the same.

\begin{lemma}\label{la:complement}
Let $G$ be a graph provided with a labeling $\lambda: V(G) \to X$ on its vertices. Then $\fer(G) = \fer(\overline{G})$.
\end{lemma}

\begin{proof}
By definition, it is enough to show that, for every feasible edge-replacement $e \to e' \in R_G$ and $\sigma \in \fer_G(e \to e')$, there is a feasible edge-replacement in $\overline{G}$ represented by $\sigma$, and vice-versa. If $e \to e' = \emptyset \to \emptyset$, then clearly $\overline{G} \cong \overline{G_{\sigma}}$ for any $\sigma \in \fer_G(\emptyset \to \emptyset) = A_G$. Hence, take $kl \to rs \in R_G^*$, and $\sigma \in \fer_G(rs \to kl)$, and observe that 
\[\overline{G} \cong \overline{G_{\sigma}} = \overline{G-v_rv_s+v_kv_l} =  \overline{G}-v_kv_l+v_rv_s,\] 
which means that $\sigma \in S_{\overline{G}}(kl \to rs)$. The converse is analogous.
\end{proof}

As a consequence, we note the following.

\begin{proposition}\label{prop:basic}
A graph $G$ is a local amoeba if and only if its complementary graph $\overline{G}$ is a local amoeba. \qed
\end{proposition}

To continue, we need some terminology.  Given a subgroup $\Gamma \leq S_X$ and $k\in X$, we denote by $\Gamma k$ the orbit of $k$ by means of the canonical action of $\Gamma$ on $X$, i.e. $$\Gamma k=\{\sigma(k)\;|\; \sigma \in \Gamma \}.$$  If, for some $Y\subseteq X$ and some $k\in Y$, we have that $\Gamma k=Y$, then we say that  $\Gamma$ \emph{acts transitively} on $Y$. We denote by ${\rm Stab}_{\Gamma}(Y)$, the \emph{stabilizer} of $\Gamma$ on $Y$, that is
 \[{\rm Stab}_{\Gamma}(Y) = \{\sigma \in \Gamma \;|\; \sigma(y) \in Y \mbox{ for all } y \in Y\}.\]
For an $x \in X$, we write  ${\rm Stab}_{\Gamma}(x)$ instead of ${\rm Stab}_{\Gamma}(\{x\})$. Moreover, for $\sigma \in S_X$ and $Y \subseteq X$, we denote with $\sigma\restrict{Y}$ the \emph{restriction} of $\sigma$ to $Y$, i.e. the mapping $\sigma\restrict{Y}: Y \to X$ defined by $\sigma\restrict{Y}(y) = \sigma(y)$ for $y \in Y$. Given sets $A, B, X, Y$ and two mappings $\sigma: A \to X$, and $\tau: B \to Y$ such that $\sigma(i) = \tau(i)$ for every $i \in A\cap B$, we denote with $\sigma \cup \tau$ the \emph{union} of $\sigma$ and $\tau$, that is, the function $\sigma \cup \tau: A \cup B \to X \cup Y$ such that
\[(\sigma \cup \tau) (i) = \left\{ \begin{array}{ll}
\sigma(i), &\mbox{if } i \in A\\
\tau(i), & \mbox{if } i \in B.
\end{array}\right.\]

According to Definition \ref{def:amoebas}, in order to determine whether a graph $G$ with a vertex labeling $\lambda: V(G) \to X$ is a local amoeba, we need to demonstrate that $\fer(G) = S_X$. Therefore, any way that enables us to generate the symmetric group $S_X$ using a small number of specific permutations would be advantageous. For example,  the following are  well known facts:
\begin{itemize}
\item  For any  $i\in[1,n]$ we have that $S_n= \left\langle (1\;2\;3 \ldots \; n), (i \; i+1)\right\rangle$.
\item The set $\{(i \; n) \; |\; i \in [n-1]\}$ generates $S_n$.
\item Let $X$ be a finite set, and let $k \in X$. Then, for any     subgroup $\Gamma \le {\rm Stab}_{S_X}(k)$ which is transitive on $X \setminus \{k\}$, we have that $S_X=\langle \Gamma \cup \{(i \; k)\}\rangle$, where $i$ is any element of $\in X \setminus \{k\}$.
\end{itemize}
Due to this, we can derive the following observation.
\begin{obs}\label{obs:daniel}
Let $G$ be a graph, and let $\lambda: V(G) \to X$ be a labeling on its vertices. In the following cases, $G$ is a local amoeba:
\begin{itemize}
\item If $X=[n]$,  $(1\;2\;3 \ldots \; n)\in \fer(G)$ and, for some  $i\in[1,n]$, also  $(i \; i+1)\in \fer (G)$.
\item If $X=[n]$,  and  $(i \; n)\in \fer (G)$ for all $i \in [n-1]$.
\item If there is a $k\in X$ such that ${\rm Stab}_{\fer(G)}(k)$ acts transitively on $X \setminus \{k\}$ and $(i\,k) \in \fer(G)$ for some $i \in X \setminus \{k\}$. 
\end{itemize}
\end{obs}

 Next, we will show how Definition \ref{def:amoebas} is very handy by completing the proof of items (iii) and (iv) of Example~\ref{ex:amoebas_basic}. We also present an example that will be used in Section \ref{sec:constr}. The different feasible edge-replacements performed in Example \ref{ex:PnandtK2} together with the corresponding permutations are illustrated Figure \ref{fig:basic examples}.

\begin{example}\label{ex:PnandtK2}
\mbox{}
\begin{enumerate}
\item[(i)] For $n \ge 1$, the path $P_n$ on $n$ vertices is a local and a global amoeba. 
\item[(ii)] For $n$ even, the graph $\frac{n}{2}K_2$ is a global amoeba but not a local amoeba.
\item[(iii)] For $n\geq 3$, the graph $C(n,1)$ obtained from a cycle on $n$ vertices by attaching a pendant vertex, is both, a local amoeba and a global amoeba.
\end{enumerate}
\end{example}

\begin{proof}
(i) For $n \le 2$, we are done by Proposition \ref{prop:neutralFER}. So we assume $n \ge 3$. Let $P = v_1v_2...v_n$ be a path on $n \geq 3$ vertices, with the labeling $\lambda:V(P) \to [n]$, $\lambda(v_i) = i$. Consider the feasible edge-replacements $n-1\;n \to 1\;n$ and $2\;3 \to 1\;3$, which give the permutations $(1\;2\;3 \ldots \; n)$ and $(1 \; 2)$. Since $ \left\langle (1\;2\;3 \ldots \; n), (1 \; 2)\right\rangle = S_n$, $P$ is a local amoeba. Let now $P' = P \cup \{v_{n+1}\}$, a graph isomorphic to $P_n \cup K_1$, with the usual labeling on its vertices. Clearly, $(1\;2\;3 \ldots \; n), (1 \; 2) \in {\rm Stab}_{\fer(P')}(n+1)$, which acts transitively on $[n] =[n+1]\setminus \{n+1\}$.  Moreover, we have now the feasible edge-replacement $n-1 \; n \to n-1 \; n+1$ that gives the permutation $(n \; n+1) \in \fer(P')$. Thus, By Observation~\ref{obs:daniel} we have that $P'$ is  a local amoeba, and thus $P$ is a global amoeba.\\
(ii) Let  $n$ be even. To show that $\frac{n}{2}K_2$ is a global amoeba, we will demonstrate that $\frac{n}{2}K_2 \cup K_1$ is a local amoeba. Let $n$ be even, and let $G$ be a graph isomorphic to $\frac{n}{2}K_2 \cup K_1$. Let $V(G) = \{v_i \;|\; i \in [n+1]\}$ and $E(G) = \{v_iv_{i+1}\;|\; i \in [n], i \mbox{ odd} \}$. As usual, let $\lambda(v_i) = i$ be a labeling on its vertices, $i \in [n+1]$. Clearly, the permutations 
$(i \; i+1)$ and $(i \; n-1)(i+1 \; n)$ are contained in $A_G \le \fer(G)$, for any odd $i \in [n]$. 
 Moreover, the edge-replacement $n-1 \; n \to n-1 \; n+1$ is certainly feasible, and so $(n \; n+1) \in \fer(G)$. Now we can see that, for any $i \in [n]$, where $i$ is odd, we have:
\[(i \; n-1)(i+1 \; n) \circ (n \; n+1) \circ (i \; n-1)(i+1 \; n) = (i+1 \; n+1) \in \fer(G).\]
Thus, we can also perform the following computation with elements from $\fer(G)$:
\[(i \; i+1) \circ (i+1 \; n+1) \circ (i \; i+1) = (i \; n+1) \in \fer(G).\]
It follows that $ (j \; n+1) \in \fer(G)$ for every $j \in [n]$. This set of permutations is known to generate $S_{n+1}$. Hence, $G$ is a local amoeba and we are done.\\
(iii) Let $C(n,1)$ be defined on the vertex set $\{v_1, v_2, \ldots, v_{n+1}\}$ with edges $\{v_iv_{i+1} \;|\; i \in \mathbb{Z}_{n}\} \cup \{ v_1v_{n+1}\}$ and let it have the labeling $\lambda:V(C(n,1)) \to [n+1]$ with $\lambda(v_i) = i$, for $i \in [n+1]$. Then $1\; (n+1) \to n \; (n+1)$ and $(n-1)\;n \to (n-1) \; (n+1)$ are feasible edge-replacements that give the permutations $(1\;2\;3 \ldots \; n)$ and $(n \; n+1)$ that generate $S_{n+1}$. Thus, $C(n,1)$ is a local amoeba and, by Proposition \ref{prop:la_degrees}, it is also a global amoeba.
\end{proof}
The graphs $P_n$, $\frac{n}{2}K_2$ and $C(n,1)$ are still quite basic examples to test to which amoeba family they belong. However, a simple application of the definition is not that easy to use when we consider a little more complicated graphs (not necessarily large ones). For example, the graph depicted below in Figure \ref{fig:G9} can be shown to be a global amoeba and that it is not a local amoeba. We will prove this later on by means of more sophisticated tools that we will develop (see Example \ref{ex:Gn-En}). Meanwhile, we leave to the reader to play around with this example.
\begin{figure}
\begin{center}
	\includegraphics[scale=0.43]{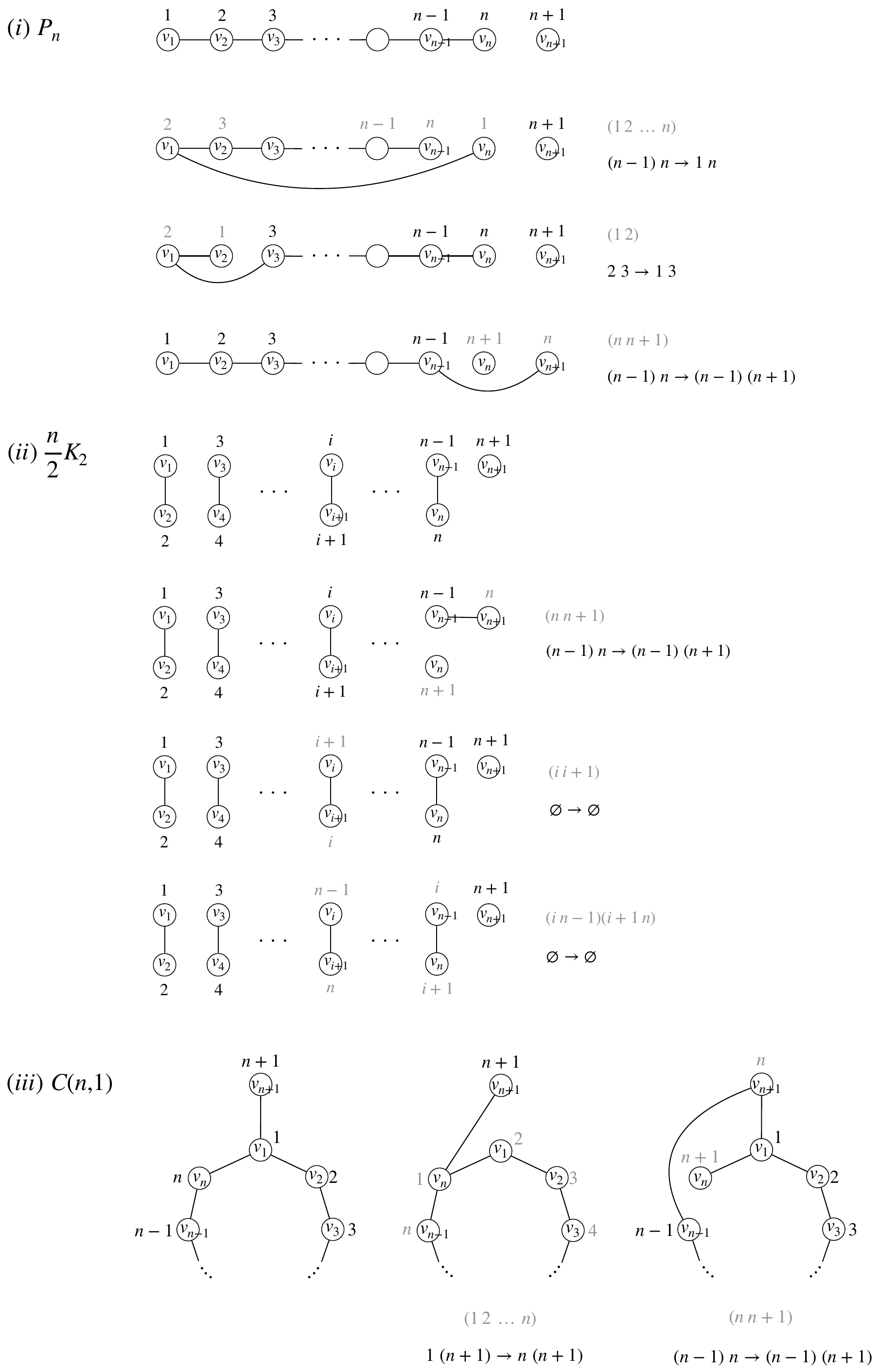}
	\caption{\small $P_n$ is local and global amoeba; $\frac{n}{2}K_2$ is global but not local amoeba; $C(n,1)$ is global and local amoeba.}
		\label{fig:basic examples}
	\end{center}
\end{figure}

\begin{figure}[H]
\begin{center}
	\includegraphics[scale=0.6]{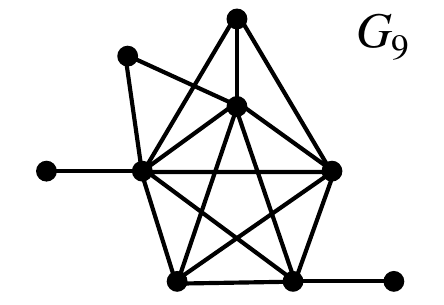}
	\caption{\small Example of a graph that is a global amoeba but not a local amoeba.} 
		\label{fig:G9}
	\end{center}
\end{figure}

We will finish this section with an interesting example that will play an important role in Section \ref{sec:e(G)X(G)w(G)}. Let $H_n$ be the graph of order $n \ge 2$ with $V(H_n) = A \cup B$ such that, taking $q=\lfloor \frac{n}{2} \rfloor$, $A = \{v_1, v_2, \ldots, v_{q}\}$ and $B = \{v_{q+1}, v_{q+2}, \ldots, v_{q+{\lceil \frac{n}{2} \rceil}}\}$, where $B$ is a clique, $A$ is an independent set and adjacencies between $A$ and $B$ are given by $v_i v_{q+j}\in E(H_n)$ if and only if $j \le i$, where $1\le i\le q$ and $1\le j\le \lceil \frac{n}{2} \rceil$ (see Figure \ref{fig:Hn}). Observe that $\deg(v_i) = i$ for all $1 \le i \le q$ and $\deg(v_{q+j}) = n - j$ for all $1 \le j \le \lceil \frac{n}{2} \rceil$. Hence, we have one vertex from each degree between $1$ and $n-1$ with exception of vertices $v_q$ and $v_n$ that have both degree $\lfloor \frac{n}{2} \rfloor$.

\begin{figure}[H]
\begin{center}
	\includegraphics[scale=0.6]{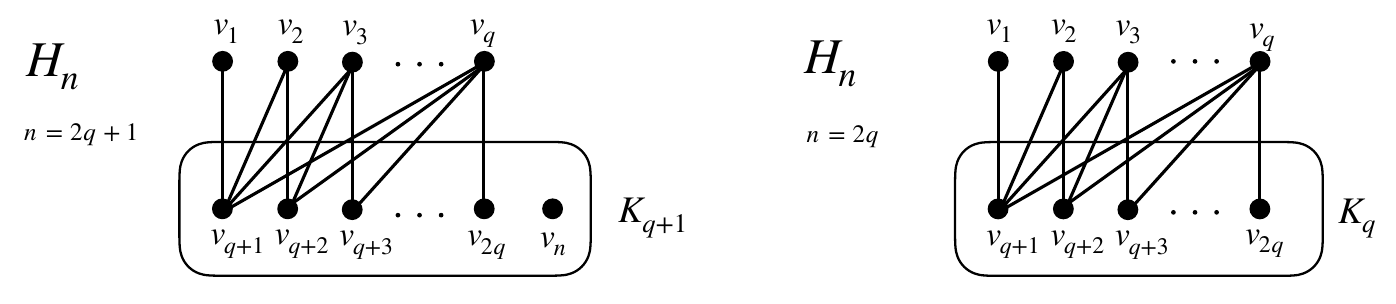}
	\caption{\small The graph $H_n$.} 
		\label{fig:Hn}
	\end{center}
\end{figure}

\begin{example}\label{ex:Hn}
$H_n$ is both a local and a global amoeba.
\end{example}

\begin{proof}
In \cite{BeCh}, it is shown that $H_n$ is the only graph of order $n$ having $\{\deg(v) \;|\; v \in V(G)\} = [n-1]$. Observe that $H_n$ can be defined recursively in the following way. By definition, $H_2 \cong K_2$, which is the same as $\overline{H_1 \cup K_1}$. Now we will show that $H_n \cong \overline{H_{n-1} \cup K_1}$ for $n \ge 3$. Indeed, this comes from the fact that the set of all degree values in $H_{n-1} \cup K_1$ is $[n-2] \cup \{0\}$ yielding that the set of all degree values in $\overline{H_{n-1} \cup K_1}$ is $\{n-1 - d \; | 0 \le d \le n-2\} = [n-1]$. Hence, $H_n \cong \overline{H_{n-1} \cup K_1}$, for each $n \ge 2$.
To show that $H_n$ is a global and a local amoeba, we proceed again by induction on $n$. $H_2 \cong K_2$  is clearly both a local and a global amoeba. Now we assume that $H_n$ is a local and a global amoeba for some $n \ge 2$. By Proposition \ref{prop:localUK1=local}, it follows that $H_n \cup K_1$ is a local amoeba. Then also $H_{n+1} \cong \overline{H_n \cup K_1}$ is a local amoeba (Proposition \ref{prop:basic}), and because it has minimum degree $1$, it is also a global amoeba (Proposition \ref{prop:localUK1=local}).
\end{proof}
\section{Main result}\label{sec:main}

\subsection{Characterizations of global amoebas}

The following theorem provides equivalent statements for the definition of a global amoeba. An important consequence of Theorem \ref{thm:eq} is that it shows that a graph $G$ is a global amoeba if and only if $G \cup K_1$ is a local amoeba, giving an answer to Question  \ref{q:t=1}.

\begin{theorem}\label{thm:eq}
Let $G$ be a non-empty graph. Let $\lambda: V(G) \to X$ be a labeling on its vertices, and let $\Gamma = \fer(G)$. For each $x \in X$, let $v_x = \lambda^{-1}(x)$. The following statements are equivalent:
\begin{enumerate}
\item[(i)] $G$ is a global amoeba.
\item[(ii)] $G\cup K_1$ is a local amoeba.
\item[(iii)] For each $x \in X$, there is a $y \in \Gamma x$ such that $\deg_G(v_y) = 1$.
\item[(iv)] For each $x \in X$ such that $\deg_G(v_x) \ge 2$, there is a $\sigma\in \Gamma$ such that $\deg_G(v_{\sigma(x)}) = \deg_G(v_x) -1$.
\end{enumerate}
\end{theorem}

Theorem \ref{thm:eq} gives us interesting information about amoebas and useful tools  to determine if a graph is a global amoeba. Before giving its proof, we will discuss some of its consequences and give some examples of how it can be applied.
 Recall that, by Proposition~\ref{prop:la_degrees}, every global amoeba $G$ has minimum degree $0$ or $1$. On the other hand, a local amoeba can have minimum degree arbitrarily large (for example $K_n$), and by Proposition~\ref{prop:localUK1=local}, a local amoeba with minimum degree $0$ or $1$ is a global amoeba, too. However, the converse of the latter is only true when the graph has isolated vertices, as can be seen with the graph $G_9$ depicted in Figure \ref{fig:G9}.

\begin{corollary}\label{co:deg=0}
Let $G$ be a graph with minimum degree $\delta = 0$. Then  $G$ is a local amoeba if and only if $G$ is a global amoeba.
\end{corollary}
\begin{proof}
If $G$ is a local amoeba with $\delta = 0$, then $G \cup K_1$ is a local amoeba by Proposition \ref{prop:localUK1=local}. Hence, it follows with Theorem \ref{thm:eq} that $G$ is a global amoeba. For the only-if part, let $G$ be a global amoeba with an isolated vertex $v$. By definition, $G-v$ is a global amoeba and thus, by Theorem \ref{thm:eq}, $G$ is a local amoeba.
\end{proof}

To prove the usefulness of Theorem \ref{thm:eq}, let us consider two examples. The graph $G_n$ we describe next is a variation of the graph $H_n$ in which we attach just a pendant vertex. It turns out that, while the property of being a global amoeba is conserved, it is no longer a local amoeba. Let $G_n$ be the graph of odd order $n = 2q+1$, for $q \ge 4$, obtained from $H_{n-1}$  (see Example \ref{ex:Hn}) by attaching a pendant vertex $v_{2q+1}$ to vertex $v_{2q}$ (see Figure \ref{fig:Gn_En}). Observe that the graph $G_9$ is precisely the graph of Figure \ref{fig:G9}. The second example is the so-called \emph{half-graph}, as was named by Erd\H{o}s and Hajnal \cite{Erdos}. The half-graph is famous for being an example that shows that  Szemer\'edi's Regularity Lemma cannot be strengthened to a regular partition \cite{CoFo}. The half-graph, which we denote here with $E_n$ is closely related to $H_n$, as can be seen in the following description: let $E_n$ be the bipartite graph of even order $n = 2q \ge 2$ with $V(E_n) = A \cup B$ such that $A = \{v_1, v_2, \ldots, v_{q}\}$ and $B = \{v_{q+1}, v_{q+2}, \ldots, v_{n}\}$, and the edges between $A$ and $B$ are given by $v_i v_{q+j}\in E(H_n)$ if and only if $j \le i$, where $1\le i, j\le q$.
\begin{example}\label{ex:Gn-En} 
The following graphs are global amoebas but not local amoebas:
\begin{enumerate}
    \item[(i)] The graph $G_n$, for odd $n \ge 9$.
    \item[(ii)] The graph $E_n$, for even $n \ge 6$.
\end{enumerate}
\end{example}
\begin{proof}
(i) Let $n = 2q+1$, where $q \ge 4$. We first note that the degrees of the vertices in $G_n$ are the following: $\deg(v_j)=j$ for $1\leq j \leq q$, $\deg(v_{q+j})=2q-j$ for $1\leq j \leq q-1$, $\deg(v_{2q})=q+1$, and $\deg(v_n)=1$. By Theorem \ref{thm:eq}(iv), to prove that $G_n$ is a global amoeba, it is enough to show that, for each $x \in [n]$ such that $\deg(v_x) \ge 2$, there is a $\sigma\in \fer(G_n)$ such that $\deg(v_{\sigma(x)}) = \deg(v_x) -1$.  For  $2\leq j \leq q$, we can see that the feasible edge replacement $j \; (q+j) \to (j-1) \; (q+j)$ implies that $(j-1 \; j)\in \fer(G_n)$. Also, for  $1\leq j \leq q-2$, the feasible edge replacement $j \; (q+j)\to j \; (q+j+1)$ implies that $(q+j \; q+j+1)\in \fer(G_n)$. Finally, the  feasible edge replacements $(q-1) \; (2q-1) \to (q-1) \; q$ and $2q \; n \to q \;n$ imply, respectively,  that $(q \; 2q-1)$ and $(q \; 2q)$ belong to $\fer(G_n)$. Therefore, every vertex with degree at least $2$ can decrease its degree in one unit, as desired. To see that $G_n$ is not a local amoeba, observe that $\fer(G_n)$ contains two orbits, $[2q]$ and $\{n\}$, because (having $n \ge 9$) there is no feasible edge-replacement that can change the role of $v_n$ (see Figure \ref{fig:Gn_En} for a  visual representation).\\
(ii) Let $n = 2q$, where $q \ge 3$. Using feasible edge-replacements similar as with $G_n$, we obtain permutations $(j \; j+1)$, and $(q+j \; q+j+1)$,  for $1 \le j \le q-1$. Moreover, the graph has exactly one automorphism represented by the permutation $(1 \; n)(2 \; n-1) \ldots (q \; q+1)$. We obtain therefore one single orbit, and, as $E_n$ has vertices of degree $1$, we conclude with Theorem \ref{thm:eq}(iii) that $E_n$ is a global amoeba. However, besides the automorphism mentioned above that interchanges the whole sets $A$ and $B$, there is no other way of permuting labels between vertices from $A$ and vertices from $B$, and thus $\fer(E_n) \not\cong S_n$ (see Figure \ref{fig:Gn_En}).
\end{proof}

\begin{figure}[H]
\begin{center}
	\includegraphics[scale=0.6]{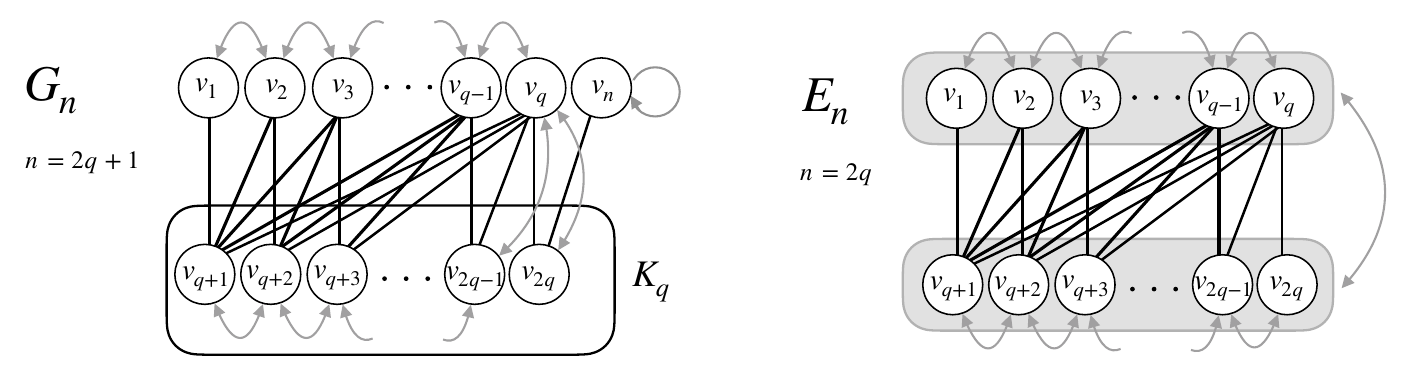}
	\caption{\small The graphs $G_n$ and $E_n$ and their orbits.} 
		\label{fig:Gn_En}
	\end{center}
\end{figure}

These examples are just a glimpse of the power of Theorem \ref{thm:eq}. In Section \ref{sec:constr}, it will prove to be a very important tool.

\subsection{Proof of Theorem \ref{thm:eq}}

To prove Theorem \ref{thm:eq}, we will need a lemma that deals with the situation of edge replacements where no isolated vertex is involved or, on the contrary, when isolated vertices are involved. The lemma is quite technical and graph theoretically it is easy to handle, but the proof of Theorem \ref{thm:eq} will require us to work with the group action, and so this language is necessary. Due to these reasons, we chose to put it in the Appendix Section as Lemma \ref{la:prop_GUtK_1}, and here we will just describe informally what it is about. First of all, let us remark that, to have a better control of the situation, we will consider a graph $G$ without isolates and we will add to it $t$ isolated vertices. Now we let $U$ be a set of $t$ vertices disjoint from $V(G)$,  for some $t \ge 1$, and $G^* = G \cup U \cong G \cup t K_1$. We use the labelings $\lambda: V(G) \to X$ and $\lambda: V(G^*) \to X \cup Y$ such that $\lambda^* \restrict{V(G)} = \lambda$, and let $\Gamma = \fer(G)$ and $\Gamma^* = \fer(G^*)$. 

Lemma \ref{la:prop_GUtK_1} has two statements. Item (i) deals with the case of a feasible edge replacement $e \to e' \in R_G$ that does not involve isolated vertices. In such a situation, we have that any permutation $\sigma \in \Gamma^*(e \to e')$ permutes labels among $X$ and maybe also among $Y$ but it does not interchange labels between the two sets. That happens, in particular, when $\sigma \in A_{G^*}$. Moreover, the restriction $\sigma  \restrict{X}$ of
every such permutation $\sigma$ corresponds to a permutation in $\Gamma(e \to e')$
and $A_G = \{\varphi\restrict{X} \;|\; \varphi \in A_{G^*}\}$.

In the second statement, we have a permutation in $\sigma \in \mathcal{E}_{G^*}$ that interchanges labels between $X$ and $Y$, meaning that there are vertices from $G$ as well as isolated vertices involved. This can only happen in one of the following situations: either we are deleting a pending edge and inserting a new edge making it pend on the same vertex but joining now one of the isolated vertices, or we have an isolated edge (two adjacent vertices of degree $1$) and we are replacing it by an edge joining two of the isolated vertices outside $G$. Item (ii) of Lemma \ref{la:prop_GUtK_1} describes then how $\sigma$ can be expressed in those two cases: it can be written as a concatenation $\sigma =  \varphi \circ (r\; k)$ or $\sigma =  \varphi \circ (r\; k)(s \; l)$, where $r,s \in X$, $k, l \in Y$, and $\varphi \in \mathcal{E}_{G^*}  \cap {\rm Stab}_{\Gamma^*}(X)$. In the second case, one can even deduce that $\varphi \in A_{G^*}$.\\ 

Now we are ready to present the proof of Theorem \ref{thm:eq}.

\begin{proof}[Proof of Theorem \ref{thm:eq}]
We will show (i) $\Rightarrow$  (iii) $\Rightarrow$  (ii) $\Rightarrow$  (i), and (iii) $\Leftrightarrow$ (iv).\\

To see (i) $\Rightarrow$  (iii), let $G$ be a global amoeba. We will first handle the case that $G$ has no isolates and then the case with isolates will easily follow.

\noindent
\emph{Case 1: Suppose that $G$ has no isolates.} \\
By definition of global amoeba, there is a $t \ge 1$ such that $G^* = G \cup tK_1$ is a local amoeba, that is,  $\fer(G^*) \cong S_{|X|+t}$.  Let $\lambda^* : V(G^*) \to X \cup Y$ a labeling on $V(G^*)$ such that $\lambda^*\restrict{V(G)} = \lambda$, and set $v_y = {\lambda^*}^{-1}(y)$, for each $y \in Y$. Let $\Gamma^* = \fer(G^*)$, and let  $x \in X$. Take a permutation $\tau \in \Gamma^*$ with $\tau(x) = m$  for some $m  \in Y$. We know that $\tau = \sigma_q \sigma_{q-1} \cdots \sigma_1$ where $\sigma_1, \cdots, \sigma_q \in \mathcal{E}_{G^*}$. Set   $\tau_i = \sigma_i \ldots \sigma_1$, $1 \le i \le q$. If $q = 1$, we have that $\tau \in \mathcal{E}_{G^*}$, and the facts that $x \in X$ and $\tau(x) \in Y$ imply that, after an edge-replacement, the vertex $v_x$ gets degree $0$, meaning that it originally had degree $1$. As  $x \in \Gamma x$, in this case  we are done. In case $q \ge 2$, we will show that $\tau$ can be chosen having the following properties:
\begin{enumerate}
\item[(a)] $\tau_i(x) \in X$ for all $1 \le i \le q-1$.
\item[(b)] $\tau_i(x) \neq \tau_j(x)$ for any pair $i,j$ with $1 \le i < j \le q$.
\item[(c)] $\sigma_i \in {\rm Stab}_{\Gamma^*}(X)$, for $1 \le i \le q-1$.
\end{enumerate}
If $\tau_i(x) \in Y$ for some $i < q$, then we can take $\tau_i$ instead of $\tau$. Hence, we may assume property (a). If $\tau_i(x) = \tau_j(x)$ for some pair $1 \le i < j \le q$, then we can take $\tau' = \sigma_q \sigma_{q-1}\ldots \sigma_{j+1} \sigma_i \sigma_{i-1} \ldots \sigma_1$ instead of $\tau$. Hence, we may assume (b). Suppose $\sigma_j \notin {\rm Stab}_{\Gamma^*}(X)$ for some $j \in \{1, 2, \ldots, q-1\}$. Choose $j$ such that it is minimum with this property. Then $\sigma_j(r) \in Y$ for some $r \in X$, say $\sigma_j(r) = k$. By (a), we have $r, k \neq \tau_{j-1}(x)$. By Lemma \ref{la:prop_GUtK_1}, either there is a $\varphi \in \mathcal{E}_{G^*} \cap {\rm Stab}_{\Gamma^*}(X)$ such that $\sigma_j = \varphi \circ (r \; k)$, or there is a $\varphi \in A_{G^*}$ such that $\sigma_j = \varphi \circ (r \; k) (s \; l)$ for certain $s \in X \setminus \{r\}$, $l \in Y \setminus \{k\}$. Suppose we have the first case.  Since $\varphi \in \mathcal{E}_{G^*}$, and $\tau_j(x) = \sigma_j(\tau_{j-1}(x)) = (\varphi \circ (r \; k) )(\tau_{j-1}(x)) = \varphi(\tau_{j-1}(x))$, we can replace $\tau$ by $\tau' = \sigma'_q \ldots \sigma'_1$, with $\sigma'_i = \sigma_i$ for $1 \le i \le q$, $i \neq j$, and $\sigma'_j = \varphi$. Suppose now  that $\sigma_j = \varphi \circ (r \; k)(s \; l)$  for certain $s \in [n]\setminus \{r\}$, $l \in Y \setminus \{k\}$, where $\varphi \in A_{G^*}$. Observe that, by item (i) of Lemma \ref{la:prop_GUtK_1}, $\varphi \in \mathcal{E}_{G^*} \cap {\rm Stab}_{\Gamma^*}(X)$. Then $\sigma_j(s) = \varphi \circ (r \; k)(s \; l) = \varphi(l) \in Y$. Thus, by (a), we know that $s \neq \tau_{j-1}(x)$.  Hence, we can proceed completely analogous  to the previous case replacing $\sigma'_j$ by $\varphi$.
Thus, we can assume that $\sigma_i \in  {\rm Stab}_{\Gamma^*}(X)$, for $1 \le i \le q-1$ and property (c) is satisfied. \\
Set $y = \tau_{q-1}(x)$. Now, since $\sigma_q \in \mathcal{E}_{G^*}$, $y \in X$,  and
\[\sigma_q(y) = \sigma_q(\tau_{q-1}(x)) =\tau_q(x) = \tau(x) = m \in Y,\] 
we have, in view of Observation \ref{obs:trivialDEG}, that $\deg_{G^*}(v_y) =1$.
Finally, we will show that $y \in \Gamma x$. To this aim, we 
will use the fact that $\sigma_i \in {\rm Stab}_{\Gamma^*}(X)$ for each $1 \le i \le q-1$ and so ${\sigma_i}\restrict{X} \in \Gamma$ by Lemma~\ref{la:prop_GUtK_1}~(i). Then ${\tau_{q-1}}\restrict{X} = (\sigma_{q-1} \ldots \sigma_1)\restrict{X} \in \Gamma$, and we have
\[ {\tau_{q-1}}\restrict{X} (x) = \tau_{q-1}(x) = y,\]
implying that $y \in \Gamma x$. Since  $\deg_G(v_y) = \deg_{G^*}(v_y) =1$, we have finished.\\

\noindent
\emph{Case 2: Suppose that $G$ has a non-empty set $U$ of isolated vertices.}\\
Observe that, by definition and in view of Proposition \ref{prop:localUK1=local}, $G - U$ is global amoeba if and only if $G$ is a global amoeba. Hence, if $G$ is a global amoeba, it follows with Case 1 that $(G - U) \cup K_1$ is a local amoeba, and, applying Proposition \ref{prop:localUK1=local} as many times as necessary, we obtain that $G$ is a local amoeba and thus $\fer(G) = S_X$. As $G$ is non-empty, it has to have at least one vertex of degree $1$ by Proposition \ref{prop:la_degrees}, say $v_y$. Hence, we have that $y \in \fer(G)x = S_Xx = X$ for every $x \in X$.\\

To see (iii) $\Rightarrow$ (ii), let $G^* = G \cup K_1$ and let $\lambda^*:V(G^*) \to X \cup \{y\}$ a labeling such that $\lambda^*\restrict{V(G)} = \lambda$. Let $\Gamma^* = \fer(G^*)$ and $X' = \{x \in X \;|\; \deg_G(v_x) = 1\}$.  Note that item (ii) is equivalent to say that $X = \bigcup_{x \in X'} \Gamma x$. For every $\sigma \in \Gamma$, consider the permutation $\sigma \cup \iota \in {\rm Stab}_{S_{G^*}}(n+1)$, where $\iota = {\rm id_{S_{\{n+1\}}}}$. Let $k \in \Gamma i$ for some $i \in X'$. Then there is a $\sigma \in \Gamma$ such that $\sigma(k) = i$.  Moreover $(i \; n+1) \in \Gamma^*$ for every $i \in X'$ because of the feasible edge-replacement $s_ii \to s_i(n+1) \in R_{G^*}$, where $v_{s_i}$ is the unique neighbor of $v_i$ in $G$. 
Then $(\sigma \cup \iota)^{-1}(i\; n+1)(\sigma \cup \iota) = (k \; n+1) \in S_{G^*}$. Since this holds for each $k \in \bigcup_{i=1}^p S_Gi = [n]$, $(k \; n+1) \in S_{G^*}$ for all $k \in [n]$ and we conclude that $S_{G^*} \cong S_{n+1}$. Hence, $G^* = G \cup K_1$ is a local amoeba.\\

Finally, the implication (ii) $\Rightarrow$ (i) is direct by the definition of global amoeba and Proposition \ref{prop:localUK1=local}, while (iii) $\Leftrightarrow$ (iv) is easy to see considering Observation \ref{obs:trivialDEG}.
\end{proof}


\section{Constructions}\label{sec:constr}

In this section, we give several constructions of global amoebas that arise from smaller ones. In particular, we will be able to construct large global amoebas, as well as global amoebas having any connected graph as one of its components and global amoeba-trees with arbitrarily large maximum degree. It is important to note that, by Proposition \ref{prop:localUK1=local} and Proposition \ref{prop:basic}, every construction given here for a global amoeba $G$ can also be used to construct a local amoeba when considering the graph $\overline{G \cup tK_1}$ for any $t \ge 1$, which is, in fact, connected. In order to simplify things, we will make use of certain abuse of notation when we deal with unions of graphs.

\begin{remark}\label{rem:groups}
Let $G = H \cup H'$ be the disjoint union of two graphs $H$ and $H'$. Let $\lambda: V(H) \to X$ and $\lambda ':V(H') \to X'$ be labelings, where $X \cap X' = \emptyset$. Let $\Gamma = \fer(H)$, and $\Gamma' = \fer(H')$. We will identify the groups $\{\sigma \cup \tau \;|\; \sigma \in \Gamma, \; \tau \in \Gamma'\}$ and $\Gamma \times \Gamma'$. Since we have that $\Gamma \times \Gamma' \le {\rm Stab}_{\fer(G)}(X) \le \fer(G)$, we have in particular that $\Gamma \cong \Gamma \times \langle {\rm id_{\Gamma'}} \rangle \le \fer(G)$ and that $\Gamma' \cong \langle {\rm id_{\Gamma}} \rangle\times  \Gamma' \le \fer(G)$. In an abuse of notation, we will say that $\Gamma$ and $\Gamma'$ are subgroups of $\fer(G)$.
\end{remark}

The formal argument of the existence of the subgroups mentioned in Remark \ref{rem:groups} can be checked in Lemma \ref{la:HUH'} of the Appendix.

\subsection{Unions and expansions}

As a consequence of Theorem~\ref{thm:eq}, we will show, in the first place, that  the vertex disjoint union of two global amoebas  is again a global amoeba.

\begin{proposition}\label{prop:HUH'isGA}
Let $H$ and $H'$ be two vertex-disjoint global amoebas. Then $G = H \cup H'$ is a global amoeba, too.
\end{proposition}

\begin{proof}
Let  $\lambda: V(H) \to X$ and  $\lambda': V(H') \to X'$ be  labelings on the vertices of  $H$ and $H'$, respectively, where $X\cap X'= \emptyset$.   
Let $I_X$ and $I_{X'}$ be the sets of all labels of the vertices of degree one in $H$ and $H'$, respectively. Let $\Gamma =\fer (H)$ and $\Gamma ' = \fer (H')$.
Since $H$ and $H'$ are global amoebas, we have, by the equivalence of items (i) and (iii) of Theorem \ref{thm:eq}, that 
\[\bigcup_{i \in I_X} \Gamma i = X \;\mbox{ and }\; \bigcup_{i \in I_{X'}} \Gamma' i = X'.\]
Hence, with $\Gamma \times \Gamma' \le \fer (G)$ (see  Remark \ref{rem:groups}), and $I=I_X \cup I_{X'}$ we obtain
\[X\cup X' = \bigcup_{i \in I} (\Gamma \times \Gamma')i \subseteq \bigcup_{i \in I} \fer (G) i,\]
from which, again by the equivalence of items (i) and (iii) of Theorem \ref{thm:eq}, we obtain that $G$ is a global amoeba.
\end{proof}

Observe that the converse statement of Proposition \ref{prop:HUH'isGA} is not valid. For example, let $H= P_k$ and $H' = C_k$, for $k \ge 3$; the graph $G =H\cup H'$ is a global amoeba as we will show in item (ii) of Example \ref{ex:GnotL}. However, $H = C_k$ is not a global amoeba (by Proposition \ref{prop:la_degrees}). We remark also at this point that there is no corresponding result to Proposition~\ref{prop:HUH'isGA} for local amoebas, since the union of two local amoebas is not necessarily again a local amoeba.   For instance, take $H=H'=P_k$, for $k \ge 2$, then, while $P_k$  is a local amoeba (see Example \ref{ex:PnandtK2}), the graph $G =H\cup H'$ is not. In item (i) of  Example~\ref{ex:GnotL} we prove something stronger. Before that, we need to prove the following proposition, which establishes that the union of several vertex-disjoint global amoebas with the same number of edges is again a global amoeba but never a local amoeba.

\begin{proposition}\label{prop:GAunion-equal-edges}
Given an integer $k \ge 2$, let $G_1, G_2, \ldots, G_k$ be connected and pairwise vertex-disjoint global amoebas such that $|E(G_i)| = |E(G_j)| \ge 1$, for $1 \le i, j \le k$. Then $G = \bigcup_{i = 1}^{k} G_i$ is a global amoeba but not a local amoeba.
\end{proposition}
\begin{proof}
Let $n = n(G)$. By Proposition \ref{prop:HUH'isGA}, $G$ is a global amoeba. However, $G$ is not a local amoeba because the only possible feasible edge replacements can just interchange edges within the components, implying that
$ \fer (G) \cong \fer (G_1) \times \ldots \times \fer (G_k) \not\cong S_n.$
\end{proof}

In the next proposition, we give a construction of a union of two vertex-disjoint graphs that is always both, a local and a global amoeba. 

\begin{proposition}
Let $G$ be a local amoeba on $n$ vertices with a vertex $v \in V(G)$ such that $\deg(v) = 1$. If $H$ is a copy of $G-v$, which is vertex-disjoint from $G$, then $G \cup H$ is both, a local and a global amoeba.
\end{proposition}
\begin{proof}
Let $\lambda: V(G) \to X=[n]$ be a labeling on $V(G)$ such that $v = v_n$, and let  $\lambda': V(H) \to Y=[n+1,2n-1]$ be a labeling on $V(H)$ such that $v_{n+i}$ is the copy of $v_i$, for $1 \le i \le n-1$. Consider now the labeling  $\lambda'': V(G\cup H) \to X\cup Y=[2n-1]$. Since $G$ is a local amoeba, we know that $\fer (G) = S_X=S_n$. By  Remark \ref{rem:groups}, we have $\fer (G) \le \fer (G\cup H)$. Thus, 
\[ (1 \, 2 \, \ldots \, n-1), ( n-1 \, n) \in \fer(G \cup H).\]
Let now $v_j$ be the neighbor of $v$ and consider the feasible edge-replacement $j \, n \to n+j \, n$, which gives the permutation
\[\sigma= (1 \, n+1)(2 \, n+2) \cdots (n-1 \, 2n-1) \in \fer(G \cup H).\]
Then we have two permutations $(1 \, 2 \, \ldots \, n-1)$ and $\sigma$ which act transitively on $[2n-1] \setminus \{n\}$. Hence, together with the permutation $ ( n-1 \, n) $, they generate $S_{2n-1}$, implying that $G \cup H$ is a local amoeba. Since $G \cup H$ has a vertex of degree $1$, it follows by Proposition \ref{prop:localUK1=local}  that $G \cup H$ is a global amoeba, too.
\end{proof}

The next proposition allows us to enlarge a global amoeba by means of taking a copy of a portion of its components where either an edge is added or deleted.

\begin{proposition}\label{prop:unions_pm_edges}
Let $G = H \cup J$ be a global amoeba, where $H$ and $J$ are vertex-disjoint subgraphs of $G$ (where $J$ can be possibly empty, meaning that $G = H$) and such that $E(G) = E(H) \cup E(J)$. Let $H'$ be a copy of $H$ which is vertex disjoint from $G$. Then we have the following facts.
\begin{enumerate}
\item[(i)] For any $e \in E(\overline{H'})$, $G \cup (H'+e)$ is a global amoeba.
\item[(ii)] For any $e \in E(H')$, $G \cup (H'-e)$ is a global amoeba.
\end{enumerate} 
\end{proposition}
\begin{proof}
We will give only the proof of item (i) as the one of (ii) can be deduced similarly.
We consider a labeling $\lambda: V(G \cup (H'+e))\to Y$ such that $\lambda(V(H)) = X$ and  $\lambda(V(H')) = X'$. As usual, we set $v_i = \lambda^{-1}(i)$, for each $i \in Y$. Moreover, for $x \in X$, let $v_{x'}$ be the copy of $v_{x}$ in $H'$ (and so we have $X' = \{x' \;|\; x \in X\}$), and let $e=v_{j'}v_{k'}$. Then $j'k' \to jk$ is a feasible edge-replacement in $G \cup (H'+e)$ and the permutation $\sigma \in S_Y$ with $\sigma(i) = i'$ and $\sigma(i') = i$, for $i \in X$, and $\sigma(i) = i$, for $i \in Y \setminus (X \cup X')$ is in $\fer (G \cup (H'+e))$. Then $x' \in \fer (G \cup (H'+e)) x$, for every $x \in X$. Since $G$ is a global amoeba, we know, by the equivalence of items (i) and (iii) of Theorem \ref{thm:eq},  that  $\fer (G)x$, and, by Remark \ref{rem:groups},  also $\fer (G \cup (H'+e)) x$, contains an element $l \in Y \setminus X'$ such that $\deg_{G \cup (H'+e)}(v_l) = \deg_G(v_l)= 1$. Hence, $G \cup (H'+e)$ is a global amoeba and we are done.
\end{proof}

The converse statements of Proposition \ref{prop:unions_pm_edges} are not valid. For item (i), we can take $G=C(k,1) \cup (C_k \cup K_1)$, where $C(k,1)$ is the graph obtained from a cycle on $k$ vertices by attaching a pendant vertex. Then $G$  is a global amoeba by Proposition \ref{prop:unions_pm_edges}(ii) because $C(k,1)$ is a global amoeba (see Example \ref{ex:PnandtK2}(iii)). However, $C_k \cup K_1$ is not a global amoeba because it violates Proposition \ref{prop:la_degrees}. On the other hand, for item (ii), consider the graph $C_k \cup P_k$, $k \ge 3$, that is a global amoeba as we will show in Example \ref{ex:GnotL}(ii). However, we also know that $P_k$ is a global amoeba (Example \ref{ex:PnandtK2}), while $C_k$ is not (by Proposition \ref{prop:la_degrees}).

\begin{example}\label{ex:GnotL}
 The following graphs are global but not local amoebas: 
   \begin{itemize}
     \item[(i)] The graph obtained by taking the disjoint union of $t$ copies of a path of order $k$, $t P_k$, for $t\geq 2$ and $k \ge 2$.
        \item[(ii)] The disjoint union of a path and a cycle of the same order $k$, $P_k\cup C_k$, for $k\geq 3$.
    \end{itemize}    
\end{example}
\begin{proof}
 (a) For $t\geq 2$ and $k\geq 2$, $tP_k$ is not a local amoeba by Proposition \ref{prop:GAunion-equal-edges}. However, the graph $tP_k$ is a global amoeba because of Proposition \ref{prop:HUH'isGA} and the fact that $P_k$ is a global amoeba (see Example \ref{ex:PnandtK2}).\\
 (b) For $k\geq 3$, the disjoint union of a path $P_k$ and a cycle $C_k$ is a global amoeba by means of Proposition \ref{prop:unions_pm_edges}(i) because $P_k$ is a global amoeba (see Example \ref{ex:PnandtK2}) and $C_k$ can be obtained from $P_k$ by adding an edge. However, $P_k\cup C_k$  is not a local amoeba. To show this, let us describe the graph with the path $v_1v_2 \dots v_k$ and the cycle $v_{k+1}v_{k+2} \dots v_{2k}v_{k+1}$. Since $P_k$ is a local amoeba, we know that the permutations corresponding to feasible edge-replacements that interchange edges with vertices in $\{v_1,v_2, \dots, v_k\}$ generate the symmetric group $S_k$. Moreover, there are no non-trivial feasible edge-replacement involving edges with  vertices  $\{v_{k+1},v_{k+2}, \dots, v_{2k}\}$ (because $C_k$ is regular). Thus, the corresponding permutations, which are only the automorphisms of $C_k$,   generate the cyclic group on $k$ elements. Finally, the other possible feasible edge-replacements are those that arise by taking one edge from the cycle and moving it to the path such that we join both end-vertices. These permutations operate by interchanging completely the sets $\{1,2, \ldots, k\}$ and $\{k+1,k+2, \ldots, 2k\}$. Thus, we cannot hope for obtaining a copy of $P_k \cup C_k$ where both the path and the cycle have vertices from both sets $\{v_1,v_2, \dots, v_k\}$ and $\{v_{k+1},v_{k+2}, \dots, v_{2k}\}$. Hence, it is not a local amoeba. 
\end{proof}
Note that item (i) of  Example~\ref{ex:GnotL} generalizes the fact that, for $n$ even, $\frac{n}{2}K_2$ is a global amoeba (see Example \ref{ex:PnandtK2}), but its proof is much more direct in the group theoretical setting than in the graph theoretical setting.

Observe that Proposition \ref{prop:HUH'isGA} and Proposition \ref{prop:unions_pm_edges} offer a wide range of possibilities for building amoebas with a diversity of components. For example, given a global amoeba $G$, the union of $G$ together with any union of graphs that arise from $G$ by adding an edge or by deleting an edge is a global amoeba. One can also include components that are built from smaller components by joining them with edges (needing possibly to apply Proposition~\ref{prop:unions_pm_edges}(i) several times). In fact, by iteratively applying Proposition \ref{prop:unions_pm_edges}, one can manage to have any connected graph $G$ as a connected component of a global amoeba, as we will show in the following corollary (see also Figure \ref{fig:conn_comp} for an illustrative drawing of the method).

\begin{theorem}\label{thm:conn_comp}
Let $G$ be any connected graph. Then there is a global amoeba $H$ having $G$ as one of its components.
\end{theorem}

\begin{proof}
We will construct a global amoeba $H$ by means of the following recursion. Let $H_0 = K_1$. For $i \ge 1$, we do the following. If $H_{i-1} \not\cong G$, then either there is one edge $e \in E(\overline{H_{i-1}})$ such that the graph $H_{i-1}+e$ is contained in $G$ as a subgraph, or there is one edge $e \in  E(\overline{H_{i-1} \cup K_1}) \setminus E(\overline{H_{i-1}})$ such that $(H_{i-1} \cup K_1) + e$ is contained in $G$ as a subgraph. In the first case we set $H_i$ to be a copy of $H_{i-1}+e$, in the second case to be a copy of $(H_{i-1} \cup K_1)+e$. Since we add in each step a new edge and the obtained graph is always contained in $G$ as a subgraph, after $m = |E(G)|$ steps, we will obtain a component $H_m \cong G$. By means of $m$ consecutive applications of Proposition \ref{prop:unions_pm_edges} (i) (where sometimes $H_{i-1}$ and sometimes $H_{i-1}\cup H_0$ plays the role of $H'$) and, since $H_0 = K_1$ is a global amoeba, it follows that $H = \bigcup_{i=0}^m H_i$ is a global amoeba having one of its components isomorphic to $G$.
\end{proof}

\begin{figure}[H]
\begin{center}
	\includegraphics[scale=0.5]{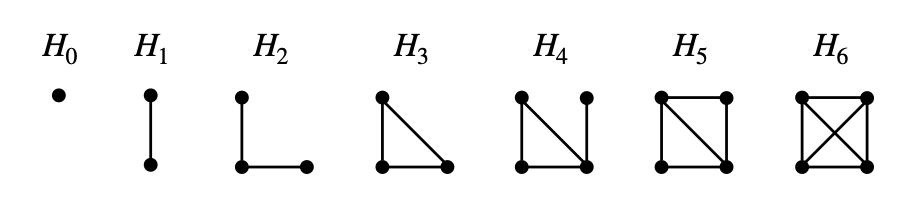}
	\caption{\small Example illustrating the proof of Theorem \ref{thm:conn_comp} with $G = K_4$.} 
		\label{fig:conn_comp}
	\end{center}
\end{figure}

As a consequence of Proposition \ref{prop:unions_pm_edges}, we obtain that there are global amoebas having arbitrarily large chromatic and clique number, but in proportion to their order these numbers may be small. In Section \ref{sec:e(G)X(G)w(G)}, we will present an example of a connected global and local amoeba whose clique and chromatic numbers equal to half its order plus one and we show that this is best possible.

\subsection{Fibonacci amoeba-trees}\label{sec:Fibo-trees}

As we know, paths, the simplest trees one can imagine having only $1$ and $2$-degree vertices, are global amoebas. In this section, we will construct an infinite family of trees via a Fibonacci-recursion which are global amoebas and which will have arbitrarily large maximum degree (and, by Proposition \ref{prop:la_degrees},  vertices of all other possible degrees).

\begin{lemma}\label{la:subgraphs}
Let $G$ be a graph provided with a labeling $\lambda: V(G) \to X$. Let $G = G' \cup G''$ for two subgraphs $G'$ and $G''$ that are not necessarily disjoint. Let $\lambda' =\lambda\restrict{V(G')}$, and $\lambda'' =\lambda\restrict{V(G'')}$ be the corresponding labelings on $G$ and $G'$, and let $\Gamma' = \fer(G')$ and $\Gamma'' = \fer(G'')$. Set $I = \lambda(V(G')) \cap \lambda(V(G''))$. Then, for every $\sigma \in \mathcal{E}_{G'} \cap \bigcap_{j \in I}{\rm Stab}_{\Gamma'}(j)$, the permutation $\sigma \cup {\rm id}_{\Gamma''}$ is in $\mathcal{E}_{G}$.
\end{lemma}

\begin{proof}
Let $\sigma \in \mathcal{E}_{G'} \cap \bigcap_{j \in I}{\rm Stab}_{\Gamma'}(j)$ and $\widehat{\sigma} = \sigma \cup {\rm id}_{\Gamma''}$, and set $v_i = \lambda^{-1}(i)$ for $i \in X$. Then there is a feasible-edge replacement $rs \to kl \in R_{G'}$ with $r,s,k,l \in \lambda(V(G'))$. This edge-replacement gives a copy $G'_{\sigma}$ of $G'$ that leaves the vertices $v_i$ with $i \in I$ untouched, i.e. $\sigma(i) = i$ for all $i \in I$. Then $G = G' \cup G'' \cong G'_{\sigma} \cup G'' = G_{\widehat{\sigma}} $. Hence, $rs \to kl$ is also a feasible edge-replacement in $G$ and $\widehat{\sigma} \in \fer_G(rs \to kl) \subseteq \mathcal{E}_G$.
\end{proof}

\begin{example}\label{ex:subgraphs}
The graph $G$ depicted below in Figure \ref{fig:G'UG''} is built by the union of the graph $G'$  and the graph $G''$, where $\lambda: V(G) \to [9]$, $v_i = \lambda^{-1}(i)$, and $\lambda(V(G')) \cap \lambda(V(G'')) = \{4,5,6\}$. Let $\Gamma' = \fer(G')$ and $\Gamma'' = \fer(G'')$. The edge-replacement $12 \to 13$ is feasible in $G'$ and we have that $\sigma = (2 \, 3) \in \fer_{G'}(12 \to 13)$. Since $\sigma = (2 \, 3) \in \mathcal{E}_{G'} \cap \bigcap_{i = 4,5,6} {\rm Stab}_{\Gamma'}(i)$, it follows by previous lemma that $\widehat{\sigma} = \sigma \cup {\rm id}_{\Gamma''} = (2 \, 3) \in \mathcal{E}_{G}$.
\end{example}

\begin{figure}[H]
\begin{center}
	\includegraphics[scale=0.4]{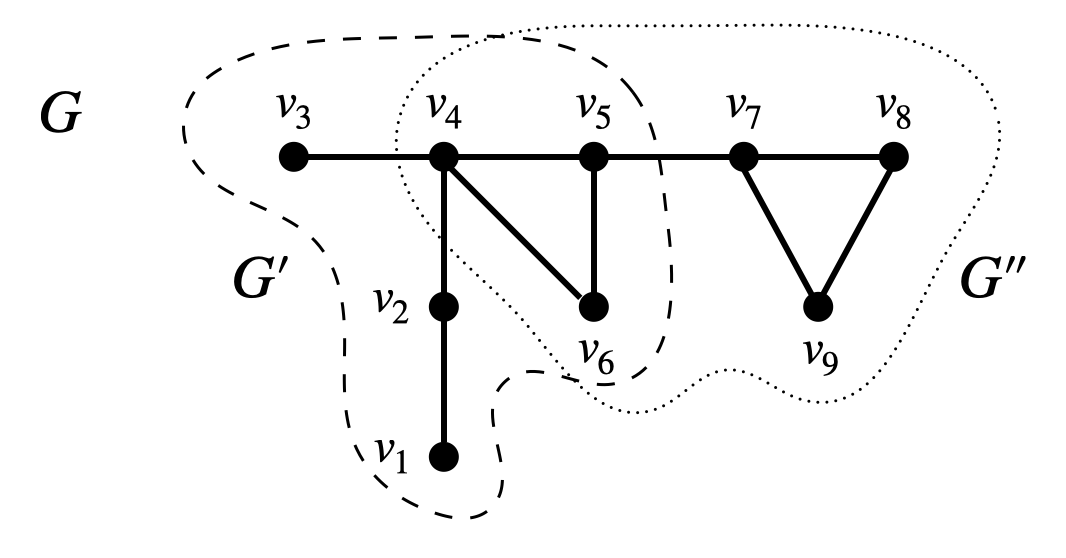}
	\caption{\small Sketch for Example \ref{ex:subgraphs}.}
	\label{fig:G'UG''}
\end{center}	
\end{figure}

Let $G$ be a graph equipped with a labeling $\lambda: V(G) \to X$, and let $v_i = \lambda^{-1}(i)$. Let $I = \{i_1, i_2, \ldots, i_k \} \subseteq X$  and let $H$ be another graph provided with a special vertex called the root of $H$. We define $G *_I H$ as the graph obtained by taking $G$ and $k$ different copies $H_1, H_2, \ldots, H_k$ of $H$ and identifying the root of $H_j$ with vertex $v_{i_j}$ of $G$, for $1 \le j \le k$ (see Figure \ref{fig:expansion}).
We will also use the following language. Given two subgraphs $G_1, G_2 \le G$ that are isomporphic, and given an isomorphism $\psi: V(G_1) \to V(G_2)$, then we say that $\psi$ \emph{induces a bijection} $\varphi: \lambda(V(G_1)) \to \lambda(V(G_2))$ by means of $\lambda(v) \mapsto \lambda(\psi(v))$, for $v \in V(G_1)$. That is, $\varphi$ is the bijection between the labels of the vertices of $G_1$ and $G_2$ corresponding to the given isomorphism.

\begin{lemma}\label{la:expansion}
Let $G$ be a graph equipped with a labeling $\lambda: V(G) \to X$. Let $H$ be another graph provided with a root. Let $I \subseteq X$, and consider the graph $G *_I H$ with labeling $\lambda^*: V(G *_I H) \to X^*$, where $\lambda^*\restrict{V(G)} = \lambda$. For each $i \in I$, let $H_i$ be the copy of $H$ attached to vertex $v_i = \lambda^{-1}(i)$, which is the root of $H_i$ in $G *_I H$. Let $X_i = \lambda^*(V(H_i))$, $i \in I$, and let $\varphi_{i, j} : X_i \to X_j$ be the bijection given by an isomorphism between $H_i$ and $H_j$ that sends $v_i$ to $v_j$, for $i,j \in I$. If $\sigma \in \mathcal{E}_G \cap {\rm Stab}_{\fer(G)}(I)$, then 
\[\sigma \cup \bigcup_{i \in I} \varphi_{i, \sigma(i)} \in \mathcal{E}_{G *_I H}.\]
\end{lemma}

\begin{proof}
Let $\sigma \in \mathcal{E}_G \cap {\rm Stab}_{\fer(G)}(I)$. Then there is a feasible-edge replacement $rs \to kl \in R_{G}$ with $r,s,k,l \in X$. This edge-replacement gives a copy $G_{\sigma}$ of $G$ such that $\sigma(i) \in I$ for all $i \in I$. Observe that the only intersections among the sets $X_i$, $i \in I$, and $X$ are given by $X_i \cap X = \{i\}$. Hence, $\widetilde{\sigma}$ is well defined as $\widetilde{\sigma}(x) = \sigma(x)$, for $x \in X$,  $\widetilde{\sigma}(x) = \varphi_{i, \sigma(i)}(x)$, for $x \in X_i$, and $\widetilde{\sigma}(i) = \sigma(i) = \varphi_{i,\sigma(i)}$, for $i \in I$. Then 
\[(G *_I H)_{\widetilde{\sigma}} = G_{\sigma} *_I H \cong G *_I H,\]
implying that  $rs \to kl$ is also a feasible edge-replacement in $G *_I H$ and thus $\widetilde{\sigma} \in \mathcal{E}_{G *_I H}$.
\end{proof}

\begin{example}\label{ex:expansion}
Let $G = v_1v_2v_3v_4v_5 \cong P_5$ and $H \cong K_{1,3}+e$, i.e. a star on three peaks together with an edge joining two of the vertices of degree $1$, where we designate one of the vertices of degree $2$ as the root of $H$. Let $\lambda: V(G) \to [5]$ with ${\lambda}(v_i) = i$, $i \in [5]$, let $I = \{2,3\}$ and let $G*_I H$ be as in Figure \ref{fig:expansion}. Let $\lambda^*:V(G *_I H) \to [11]$ such that ${\lambda^*}(v_i) = i$. Then $4\, 5 \to 1 \, 5 \in R_G$ with $\sigma = (1 \, 4)(2 \, 3) \in \fer_G(4\, 5 \to 1 \, 5)$, and $\varphi_{2,3} = \varphi_{3,2} = (2 \, 3)(6 \, 9)(7 \, 10)(8 \, 11)$. Since $\sigma \in \mathcal{E}_G \cap  {\rm Stab}_{\fer(G)}(\{2,3\})$, it follows by Lemma \ref{la:expansion} that $\widetilde{\sigma}=(1 \, 4) (2 \, 3)(6 \, 9)(7 \, 10)(8 \, 11) \in \mathcal{E}_{G *_I H}$.
\end{example}

\begin{figure}[H]
\begin{center}
	\includegraphics[scale=0.4]{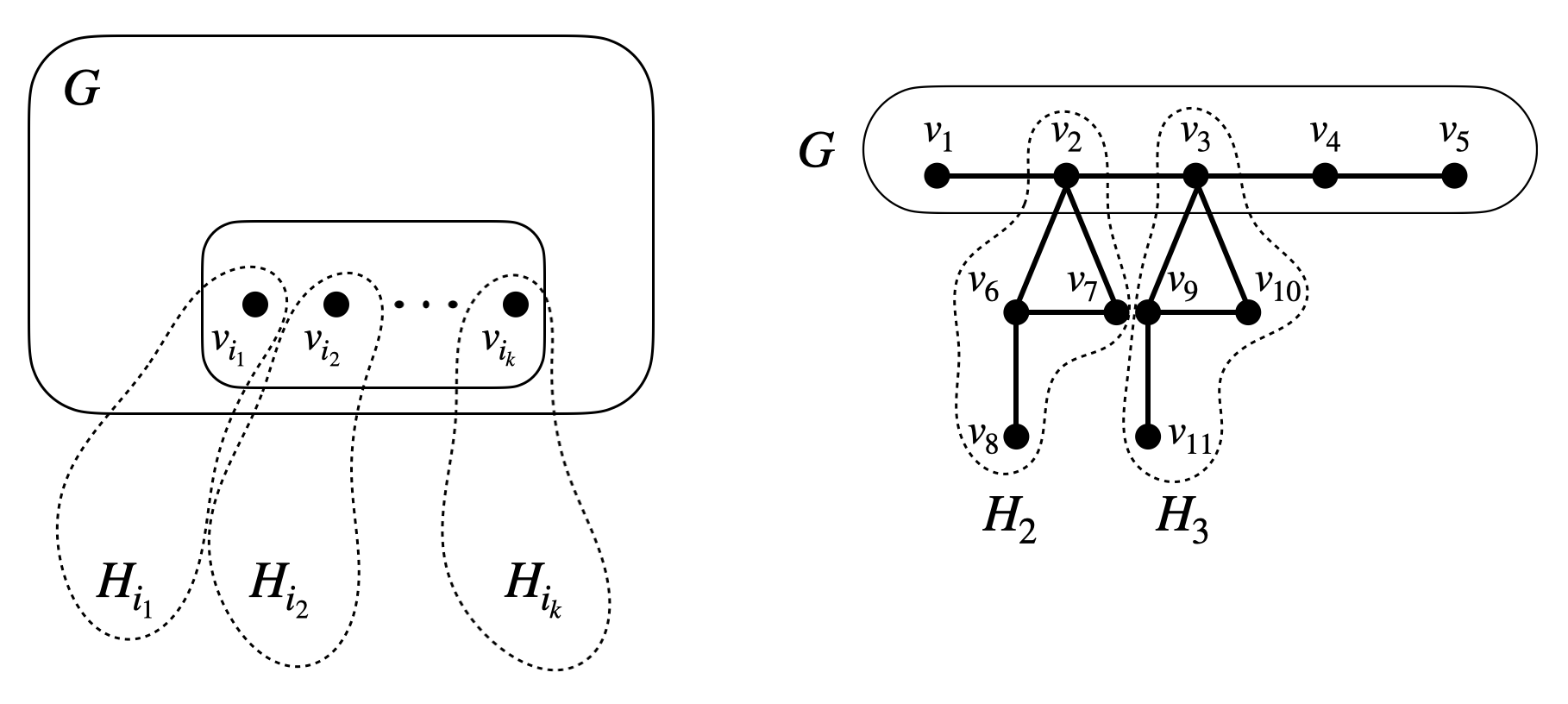}
	\caption{\small Sketch of a graph $G *_I H$, where $I = \{i_j \;|\; j \in [k]\}$, and of Example \ref{ex:expansion}.}
	\label{fig:expansion}
\end{center}	
\end{figure}

We will describe a family of trees that are constructed via a Fibonacci recursion. We define $T_1 = T_2 = K_2$. For $i \ge 2$, we define $T_{i+1}$ as the tree consisting of one copy $T$ of $T_{i-1}$ and one copy $T'$ of $T_i$, where a vertex of maximum degree of $T$ is joined to a vertex of maximum degree of $T'$ by means of a new edge, see Figure \ref{fig:Fib_amoebas}. Observe that $\Delta(T_i) = i - 1$ for $i \ge 2$, while $n(T_i) = 2F_i$, being $F_i$ the $i$-th Fibonacci number. Note also that, for $i \ge 4$, $T_i$ has only one vertex of maximum degree, which we will call the \emph{root} of $T_i$. For the case that $i \le 3$, we will designate one of the vertices of maximum degree as the root of $T_i$ and this will be the vertex that will be used to attach the new edge in the construction of $T_{i+1}$.

\begin{figure}[H]
\begin{center}
	\includegraphics[scale=0.3]{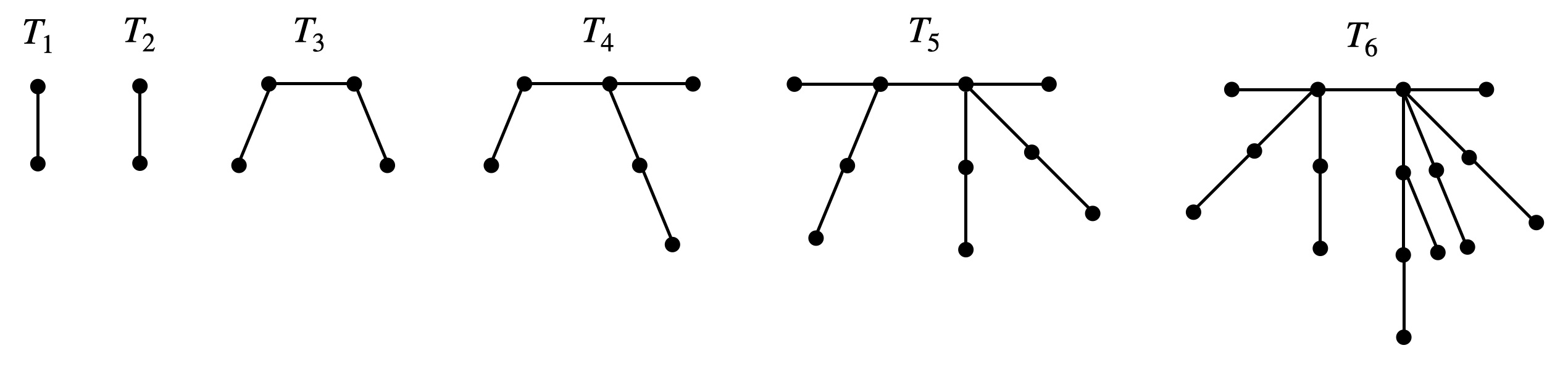}
	\caption{\small Fibonacci amoeba-trees $T_i$, $1 \le i \le 6$.}
	\label{fig:Fib_amoebas}
\end{center}	
\end{figure}

\begin{theorem}\label{thm:Ti_GA}
$T_i$ is a global amoeba for all $i \ge 1$.
\end{theorem}

\begin{proof}
Let $T$ be a tree isomorphic to $T_i$ equipped with a labeling $\lambda:V(T) \to X$, and let $v_i = \lambda^{-1}(i)$, and $\Gamma = \fer(T)$. Let $c \in J$ such that $v_c$ has maximum degree in $T$. We will show by induction on $i$ that there is a subset $\mathcal{R} \subseteq \mathcal{E}_T \cap {\rm Stab}_{\Gamma}(c)$ such that $\langle \mathcal{R} \rangle$ acts transitively on $X \setminus \{c\}$.

If $i = 1,2$, there is nothing to prove. If $i = 3$, then $T \cong P_4$, say $T = v_4v_3v_1v_2$ with $c = 1$. Then the feasible edge-replacements $34 \to 24$ and $13 \to 14$ give respectively the permutations $(2 \, 3)$ and $(3 \, 4)$, which act transitively on $\{2,3,4\} = X \setminus \{c\}$. If $i = 4$, then let $T$ be the tree built from the path $v_4v_3v_1v_2 \cong T_3$ and a $T_2 \cong K_2$, given by $v_5v_6$, and the edge $v_1v_5$ joining both trees. Clearly, the only maximum degree vertex is $v_1$ and thus $c = 1$. Then the feasible edge-replacements $34 \to 24$ and $13 \to 14$ give respectively the permutations $(2 \, 3)$ and $(3 \, 4)$, which together with the automorphism  $(3 \, 5) (4 \, 6)$, act transitively on $[6] \setminus \{1\}$ leaving $c = 1$ fixed. 

Now suppose that $i \ge 4$ and that we have proved the above statement for integer values at most $i$. Let $T \cong T_{i+1}$. For a subset $J \subset X$, we define $V_J = \{v_j \;|\; j \in J\}$, $T_J = T[V_J]$, and $\Gamma_J = \fer(T_J)$ using the inherited labeling $\lambda\restrict{V_J}$. Let $X = U \cup W$ be a partition of $X$ such that $T_U \cong T_{i-1}$ and $T_W \cong T_i$. Further, let $U = A \cup B$ and $W = C \cup D$ be partitions such that $T_A \cong T_{i-3}$, $T_B \cong T_{i-2}$, $T_C \cong T_{i-1}$, and $T_D \cong T_{i-2}$. By construction, $v_c$ is the root of $T_C$. Let $a, b, d \in X$ be such that $v_a, v_b, v_d$ are the roots of $T_A, T_B$, and $T_D$, respectively.  Notice that $v_av_bv_cv_d$ is a path of length $4$ in $T$. See Figure \ref{fig:TreeT} for a sketch.

\begin{figure}[H]
\begin{center}
	\includegraphics[scale=0.5]{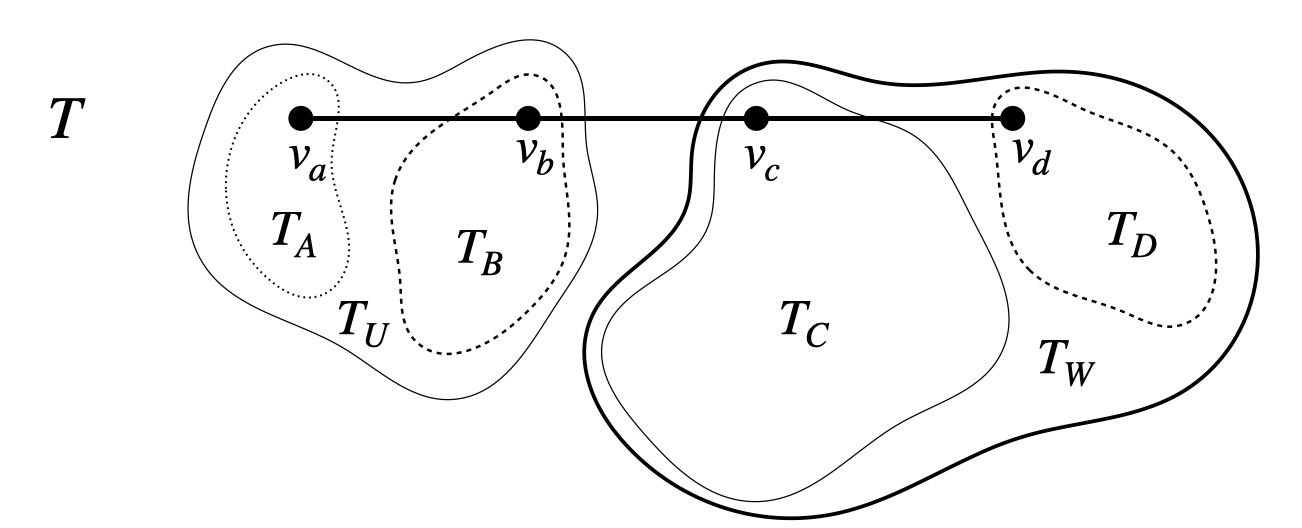}
	\caption{\small Sketch of the tree $T \cong T_{i+1}$ with its subtrees $T_U \cong T_{i-1}$ and $T_W \cong T_i$, and subsubtrees $T_A \cong T_{i-3}$, $T_B \cong T_{i-2}$, $T_C \cong T_{i-1}$, and $T_D \cong T_{i-2}$.}
	\label{fig:TreeT}
\end{center}	
\end{figure}

By the induction hypothesis, there are subsets $\mathcal{R}_U \subseteq \mathcal{E}_{T_U} \cap {\rm Stab}_{\Gamma_U}(b)$ and $\mathcal{R}_W \subseteq \mathcal{E}_{T_W} \cap {\rm Stab}_{\Gamma_W}(c)$ such that $\langle \mathcal{R}_U \rangle$ acts transitively on $U \setminus \{b\}$ and $\langle \mathcal{R}_W \rangle$ acts transitively on $W \setminus \{c\}$. Let $\widehat{\mathcal{R}}_U = \{\sigma \cup {\rm id}_{\Gamma_W} \;|\; \sigma \in \mathcal{R}_U\}$ and $\widehat{\mathcal{R}}_W = \{\sigma \cup {\rm id}_{\Gamma_U} \;|\; \sigma \in \mathcal{R}_W\}$. Then, by Lemma~\ref{la:subgraphs}, $\widehat{\mathcal{R}}_U, \widehat{\mathcal{R}}_W \subseteq \mathcal{E}_{T}$. Moreover, the transitive action is inherited, i.e., $\langle \widehat{\mathcal{R}}_U \rangle$ acts transitively on $U \setminus \{b\}$ and $\langle \widehat{\mathcal{R}}_W \rangle$ acts transitively on $W \setminus \{c\}$.

Consider now the tree $T(B,D)$ that is obtained by identifying all vertices from $V_B$ with vertex $v_b$ and all vertices from $V_D$ with vertex $v_d$, i. e. we contract the sets $V_B$ and $V_D$ each into a single vertex (see Figure \ref{fig:T(B,D)&T(U,C)}). Observe  that $ab \to ad$ is a feasible edge-replacement in $T(B,D)$ with $(b \, d) \in \fer_{T(B,D)}(ab \to ad)$, and that $T \cong T(B,D) *_{\{b,d\}} T_{i-2}$. Since $T_B \cong T_D \cong T_{i-2}$, there is a bijection $\varphi: B \to D$ such that $\varphi(b) = d$ given by an isomorphism between $T_B$ and $T_D$ that maps $v_b$ to $v_d$. Then, by Lemma \ref{la:expansion}, we have that $ab \to ad$ is a feasible edge-replacement in $T$ with $\tau \in \fer_T(ab \to ad)$ defined by 
\[\tau = (b \, d) \cup \varphi \cup \varphi^{-1} \cup {\rm id}_{\Gamma_A} \cup {\rm id}_{\Gamma_C},\] 
and such that  $\tau\in \mathcal{E}_{T}$.
Moreover, $\tau$ leaves $c$ fixed and so $\tau \in \mathcal{E}_T \cap {\rm Stab}_{\Gamma}(c)$. Now we define
\[\mathcal{R} = \widehat{\mathcal{R}}_U \cup \widehat{\mathcal{R}}_W \cup \{ \tau \}.\]
Since $\langle\widehat{\mathcal{R}}_U \rangle$ acts transitively on $U \setminus \{b\}$, and $\langle \widehat{\mathcal{R}}_W \rangle$ acts transitively on $W \setminus \{c\}$, these two sets together with  $\tau$ generate a group $\langle \mathcal{R} \rangle$ that acts transitively on $X \setminus \{c\}$. 

Hence, we have shown that if $T \cong T_i$, for any $i \ge 1$, then there is a subset $\mathcal{R} \subseteq \mathcal{E}_T \cap {\rm Stab}_{\Gamma}(c)$ such that $\langle \mathcal{R} \rangle$ acts transitively on $X \setminus \{c\}$, where $c \in X$ such that $v_c$ has maximum degree in $T$.

To finish the proof, we will show that there is a permutation $\rho \in \Gamma$ such that $\langle \mathcal{R} \cup \{\rho\} \rangle$ acts transitively on $X$, meaning that $\Gamma$ acts transitively on $X$, too, which implies that $T$ is a global amoeba by Theorem \ref{thm:eq}(iii). Since we know already that $\langle \mathcal{R} \rangle$ acts transitively on $X \setminus \{c\}$, we just need to find a $\rho\in \Gamma$ with $\rho(c) \neq c$. Indeed, there is such a permutation $\rho$, namely one produced by the feasible edge-replacement $cd \to bd$ in $T$, which, by Lemma \ref{la:expansion}, can be obtained by means of the permutation $(b \, c) \in \fer_{T(U,C)}(cd \to bd)$ through 
\[\rho = (b \, c) \cup \psi \cup \psi^{-1} \cup {\rm id}_{\Gamma_W},\] 
where $\psi: C \to U$ is the bijection with $\psi(c) = b$ given by an isomorphism between $T_C$ and and $T_U$ that sends $v_c$ to $v_b$. Hence, $T_i$ is a global amoeba for all $i \ge 1$. 
\end{proof}

\begin{figure}[H]
\begin{center}
	\includegraphics[scale=0.4]{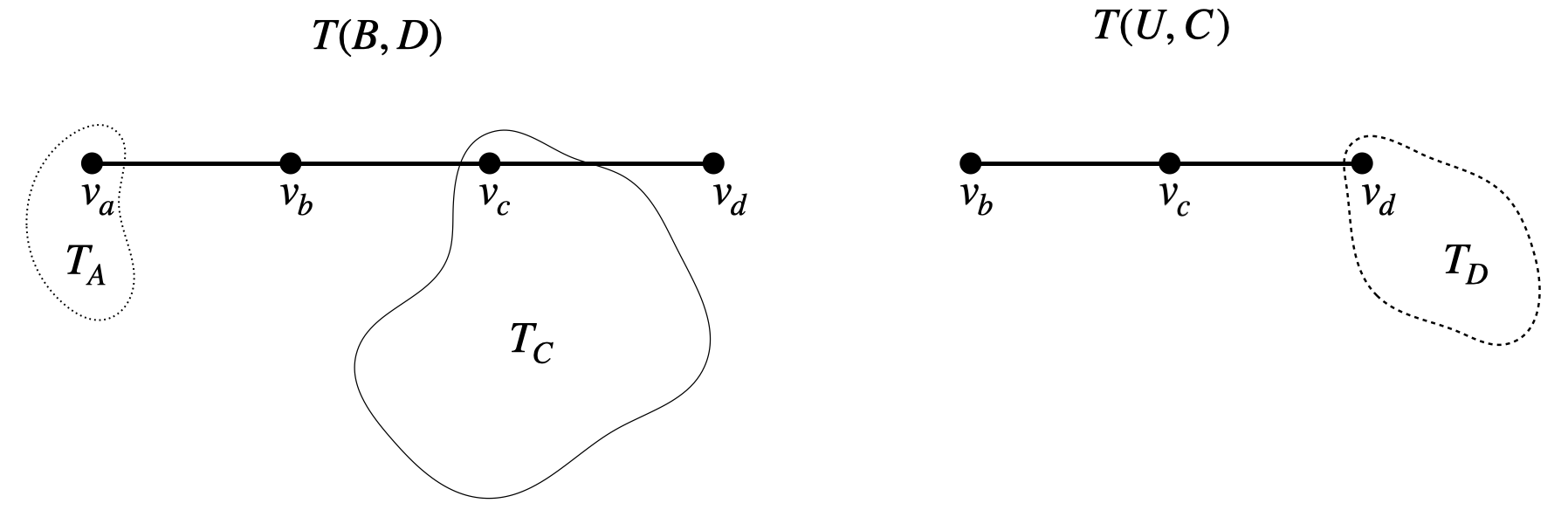}
	\caption{\small Trees $T(B,D)$ and $T(U,C)$.}
	\label{fig:T(B,D)&T(U,C)}
\end{center}	
\end{figure}

That $T_1$, $T_2$ and $T_3$ are local amoebas follows directly from Example \ref{ex:PnandtK2}. To see that $T_5$ is also a local amoeba, let  $T_5$ have vertices $v_i$, $1 \le i \le 10$, distributed as in Figure \ref{fig:T5_tr_action}, and consider the feasible edge-replacements $9 \, 10 \to 8 \, 10$, $5 \, 6 \to 2 \, 6$, $1 \, 7 \to 1 \, 9$, and $1 \, 5 \to 7 \, 5$ that produce the permutations $(8 \; 9)$, $(2 \; 5)$, $(7 \; 9)(8 \; 10)$, and $(1 \; 7)(2 \; 8)(3 \; 9)(4 \; 10)$. It is not difficult to check that, these permutations act transitively on $[10] \setminus \{6\}$ (see Figure \ref{fig:T5_tr_action} for a visual representation of this partial orbit). Finally, consider the feasible edge-replacement $1 \, 5 \to 1 \, 6$ that gives the permutation $(5 \; 6)$, which together with the above $4$ permutations, generate $S_{10}$ by Observation  \ref{obs:daniel}. A very similar argument can be applied for $T_4$ to show that it is a local amoeba. Hence, $T_i$ is a local amoeba for all $1 \le i \le 5$. However, for $i \ge 6$, the above argument cannot be generalized that simply. We leave this as an open question (see Problem \ref{probl:Ti_LA}).

\begin{figure}[H]
\begin{center}
	\includegraphics[scale=0.65]{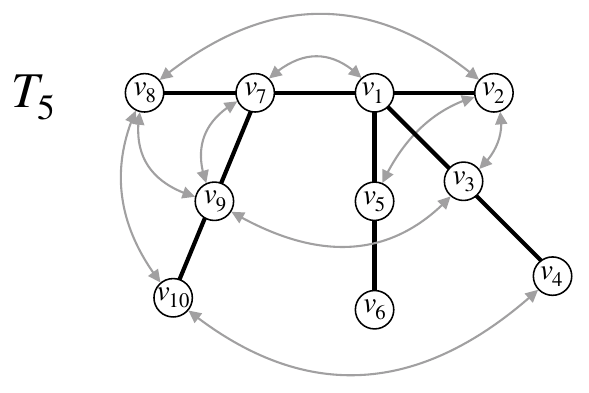}
	\caption{\small Graph $T_5$ with the transitive action on $[10]\setminus \{6\}$.}
		\label{fig:T5_tr_action}
	\end{center}
\end{figure}


\section{Extremal  global amoebas with respect to size, chromatic number and clique number}\label{sec:e(G)X(G)w(G)}

We denote by $e(G)$, $\chi(G)$ and $\omega(G)$ the \emph{size} (number of edges), the \emph{chromatic number} (smallest number of colors in a proper vertex coloring) and the \emph{clique number} (order of a maximum clique) of $G$ respectively.

We shall note that the degrees constraint established in Proposition~\ref{prop:la_degrees} compromises the number of edges that a global amoeba or a local amoeba with small minimum degree can have. In this section, we will show that a graph of order $n$ that is  a global amoeba cannot have more than $\lfloor \frac{n^2}{4}\rfloor$ edges. Interestingly, it turns out that this bound is sharp. We will also prove that the chromatic number, and thus the clique number, of a global amoeba of order $n$ can not be greater than $\lfloor \frac{n}{2} \rfloor+1$. Again, this upper bound is sharp and we will prove that it is reached when having the maximum possible number of edges.

The family of graphs that proves the sharpness in the upper bounds mentioned in the previous paragraph is $H_n$, which was given in Example \ref{ex:Hn} as the graph of order $n$ with $V(H_n) = A \cup B$ such that, taking $q=\lfloor \frac{n}{2} \rfloor$,  $A = \{v_1, v_2, \ldots, v_{q}\}$ and $B = \{v_{q+1}, v_{q+2}, \ldots, v_{q+{\lceil \frac{n}{2} \rceil}}\}$, where $B$ is a clique, $A$ is an independent set and adjacencies  between $A$ and $B$ are given by $v_i v_{q+j}\in E(H_n)$ if and only if $j \le i$, where $1\le i\le q$ and $1\le j\le \lceil \frac{n}{2} \rceil$. It was shown in the mentioned example that $H_n$ is both a global and a local amoeba. It is also very simple to note that $e(H_n) = \lfloor \frac{n^2}{4}\rfloor$ and  $\omega(H_n) = \lfloor \frac{n}{2}\rfloor+1$. 

The next theorem gives upper bounds for the edge number $e(G)$, the chromatic number $\chi(G)$, and the clique number $\omega(G)$, of a global amoeba with minimum degree $1$. We will use the Powell-Welsh bound on the chromatic number of a graph $G$ \cite{PoWe} (see \cite{Bon69} for an alternative proof):
\begin{equation}\label{eq:Bon}
\chi(G) \le \max_{1 \le i \le n} \min\{d_i+1, i\}
\end{equation}
where $d_1 \ge d_2 \ge \ldots \ge d_n$ is the degree sequence of $G$.

\begin{theorem}\label{thm:GA_max_e(G)X(G)w(G)}
If $G$ is a global amoeba of order $n$ with minimum degree $\delta(G) = 1$, then 
\begin{itemize}
\item[(i)] $e(G) \le \lfloor \frac{n^2}{4}\rfloor$, and
\item[(ii)] $\omega(G) \le \chi(G) \le \lfloor \frac{n}{2}\rfloor+1$,
\end{itemize}
where all bounds are sharp. Moreover, we have the following relations concerning the equalities in the above bounds.
\begin{itemize}
\item[(iii)] If $e(G) = \lfloor \frac{n^2}{4}\rfloor$ then $\omega(G) = \chi(G) = \lfloor \frac{n}{2}\rfloor+1$, but the converse is not true.
\item[(iv)] We have $\omega(G) = \lfloor \frac{n}{2}\rfloor+1$ if and only if  $\chi(G) = \lfloor \frac{n}{2}\rfloor+1$.
\end{itemize}
\end{theorem}

\begin{proof}
Let $d_1 \ge d_2 \ge \ldots \ge d_n =1$ be the degree sequence of a global amoeba $G$, and let $D = \{d_i \;|\; i \in [n] \}$. By Proposition \ref{prop:la_degrees}, we know that $D = [d_1]$ where $d_1 \in [n-1]$ and that, for every $i\in[n]$, 
\begin{equation}\label{eq:d_i}
d_i\leq n+1-i.
\end{equation}
Now we will prove the four items separately.\\

\noindent
(i) By inequality (\ref{eq:d_i}),  the sum of the $\lfloor \frac{n}{2}\rfloor$ smallest degrees satisfies
\begin{equation}\label{eq:sum}
\sum_{i=\lceil \frac{n}{2}\rceil+1}^n d_i \le \sum_{i=\lceil \frac{n}{2}\rceil+1}^n  n+1-i =  \sum_{i=1}^{\lfloor \frac{n}{2}\rfloor} i = \frac{1}{2} \left\lfloor \frac{n}{2}\right\rfloor\left(\left\lfloor \frac{n}{2}\right\rfloor+1\right).
\end{equation}

Let $L$ be the set of the $\lceil \frac{n}{2}\rceil$ vertices having the largest degrees (corresponding to the degrees $d_1, d_2, \ldots, d_{\lceil \frac{n}{2}\rceil}$), and let $S = V(G) \setminus L$. Denote by $e(L)$ the number of edges induced by the vertices in $L$ and by $e(L,S)$ the number of edges between $L$ and $S$. Then we have
\begin{equation}\label{eq:ErdGal2}
\sum_{i=1}^{\lceil \frac{n}{2}\rceil}d_i  = 2 e(L) + e(L, S) \leq \left\lceil \frac{n}{2}\right\rceil\left(\left\lceil \frac{n}{2}\right\rceil-1\right) +\sum_{i=\lceil \frac{n}{2}\rceil+1}^n d_i.
\end{equation}
Hence, inequalities (\ref{eq:sum}) and (\ref{eq:ErdGal2}) yield
\begin{align*}
2e(G) = \sum_{i=1}^{n}d_i & \leq \left\lceil \frac{n}{2}\right\rceil\left(\left\lceil \frac{n}{2}\right\rceil-1\right) +2 \sum_{i=\lceil \frac{n}{2}\rceil+1}^n d_i \\
&\leq \left\lceil \frac{n}{2}\right\rceil\left(\left\lceil \frac{n}{2}\right\rceil-1\right) +\left\lfloor \frac{n}{2}\right\rfloor\left(\left\lfloor \frac{n}{2}\right\rfloor+1\right)=\left\lfloor \frac{n^2}{2}\right\rfloor,
\end{align*}
and the bound follows because $\frac{1}{2}\lfloor \frac{n^2}{2}\rfloor =\lfloor \frac{n^2}{4}\rfloor$. Observe that the bound is attained for the graph $H_n$ shown in Example \ref{ex:Hn} to be global amoeba.\\

\noindent(ii)  For every $\left\lfloor \frac{n}{2} \right\rfloor+2\leq i\leq n$, we have, using (\ref{eq:d_i}) with $i = \left\lfloor \frac{n}{2}\right\rfloor+2$, that
$$\min \{d_i +1,i\}\le d_i+1\leq d_{\left\lfloor \frac{n}{2}\right\rfloor+2}+1\leq \left\lceil \frac{n}{2} \right\rceil\leq  \left\lfloor \frac{n}{2} \right\rfloor+1.$$
For the remaining cases $1\leq i\leq \left\lfloor \frac{n}{2} \right\rfloor+1$, we obtain as well
$$\min \{d_i +1,i\}\leq i\leq \left\lfloor \frac{n}{2} \right\rfloor+1.$$ 
Altogether it follows with (\ref{eq:Bon}), that $\chi(G) \le \lfloor \frac{n}{2}\rfloor+1$.
The trivial inequality $\omega(G) \le \chi(G)$ yields the result. To show that the bound is sharp, consider again the graph $H_n$, that is a global amoeba by Example \ref{ex:Hn}, and that satisfies $\omega(H_n) = \chi(H_n) = \lfloor \frac{n}{2} \rfloor + 1$.\\

\noindent
(iii) Observe now that a global amoeba $G$ with degree sequence $d_1 \ge d_2 \ge \ldots \ge d_n=1$ satisfies $e(G)=\lfloor \frac{n^2}{4}\rfloor$ if and only if   equalities in (\ref{eq:sum}) and (\ref{eq:ErdGal2})  hold. This happens if and only if, on the one hand, the smallest degrees $1, 2, \ldots, \lfloor\frac{n}{2} \rfloor -1$ appear each one once (while the degree $\lfloor\frac{n}{2} \rfloor$ appears at least once) and, on the other hand, the sum of the degrees of the  $\lceil \frac{n}{2} \rceil$ vertices having the largest degrees is exactly $\left\lceil \frac{n}{2}\right\rceil\left(\left\lceil \frac{n}{2}\right\rceil-1\right) +\sum_{i=\lceil \frac{n}{2}\rceil+1}^n d_i$, meaning that they form a clique and that the complementary set (i.e. the $\lfloor \frac{n}{2} \rfloor$ vertices of the smallest degrees) is edge-less. From here, it is easy to see that $\omega(G) = \lfloor \frac{n}{2}\rfloor +1$. Thus, by item $(ii)$, also $\chi(G) = \lfloor \frac{n}{2}\rfloor +1$. To see that the converse is not true, take the graph $G_n$ defined before Example \ref{ex:Gn-En} that is a global amoeba with minimum degree $1$, and $\omega(G) = \chi(G) = \lfloor \frac{n}{2}\rfloor +1$, but has $e(G_n)   < \lfloor \frac{n^2}{4}\rfloor$ for $n \ge 4$. \\

\noindent
(iv) The necessity part is clear because of item (ii). For the converse, suppose that $\chi(G) = \lfloor \frac{n}{2}\rfloor+1$. Let $L$ be the set of all vertices of degree at least $\lfloor \frac{n}{2}\rfloor$. Since $V(G) \setminus L$ contains all vertices of degree at most $\lfloor \frac{n}{2}\rfloor -1$, it follows by Proposition \ref{prop:la_degrees} that
\[\left\lfloor \frac{n}{2}\right\rfloor -1 \le |V(G) \setminus L| = n - |L|.\]
Hence, we obtain that $|L|  \le \lceil \frac{n}{2}\rceil + 1$. 

We assume first that $n$ is even and we suppose for a contradiction that $L$ is not a clique. Then we can color the vertices of $L$ with $|L|-1 = \frac{n}{2}$ different colors such that there are no adjacent vertices with the same color.  Since the vertices in $V(G) \setminus L$ have degree not larger than  $\frac{n}{2} -1$, we can proceed coloring the vertices of $V(G) \setminus L$ one after the other by taking always one of the colors that is not already taken by one of its neighbors. In this way, we use at most $\frac{n}{2}$ colors and there are no two adjacent vertices with the same color, implying the contradiction $\chi(G) \le \frac{n}{2}$. Hence, $L$ has to be a clique and it follows that $\omega(G) = \frac{n+1}{2}$.

Let now $n$ be odd. Let $v \in V(G) \setminus L$ be the vertex of degree $\frac{n-1}{2}$. If $|N(v) \cap L| = \frac{n-1}{2}$ and $N(v) \cap L$ is a clique, we have finished because then $N[v]$ is a clique and so $\omega(G) = \frac{n+1}{2} = \lfloor \frac{n}{2}\rfloor +1$. Hence, we may assume that $N(v) \cap L$ is not a clique or that $|N(v) \cap L| \le \frac{n-3}{2}$. In both cases we can color the vertices of $L$ with $\frac{n-1}{2}$ different colors in such a way that there is no adjacent pair with the same color but taking also care that there are no more than $\frac{n-3}{2}$ colors in $N(v) \cap L$. Now we can color $v$ using a color that has been used in $L \setminus N(v)$. The remaining vertices have degree at most $\frac{n-3}{2}$, so that we can proceed in a greedy way as in previous case using no more than $\frac{n-1}{2}$ colors in total. Hence, it follows that $\chi(G) \le \frac{n-1}{2} = \lfloor \frac{n}{2}\rfloor$, a contradiction.
\end{proof}

The search for the extremal family in the bound of item (i) of Theorem \ref{thm:GA_max_e(G)X(G)w(G)} requires a much more detailed analysis  that takes into account, not only the degree sequence, but the inner structure of a global amoeba. We already know that, if the graph has maximum degree $n-1$, the only extremal graph is $H_n$ (see proof of Example  \ref{ex:Hn}). However, if the maximum degree is smaller, there may be different possibilities for the repetitions among the higher degrees. Still, we believe that the only possible graph attaining equality here is $H_n$ (see Conjecture~\ref{conj:max_e(G)}).

We finish this section with a simple upper bound on the maximum degree of a global amoeba with minimum degree $1$.
\begin{proposition}\label{prop:max_deg}
Let $G$ be a global amoeba on $n$ vertices and $m$ edges such that $\delta(G) = 1$. Then  
\[\Delta(G) \le \frac{1}{2} \left(1 + \sqrt{1- 8n + 16 m}\right) < 1 + 2\sqrt{m},\]
and the left inequality is sharp. 
\end{proposition}
\begin{proof}
Let $\Delta = \Delta(G)$. By Proposition \ref{prop:la_degrees}, we deduce that $2m \ge (n - \Delta)+\sum_{i = 1}^\Delta i$,
which gives $4m \ge \Delta^2 - \Delta + 2n$. Solving the quadratic inequality, the bound 
\[\Delta(G) \le \frac{1}{2} \left(1 + \sqrt{1- 8n + 16 m}\right)\] 
follows. The bound is sharp for the star forest $K_{1,2} \cup K_{1,3} \cup \ldots \cup K_{1, \Delta}$, which can be shown to be a global amoeba by means of Proposition \ref{prop:unions_pm_edges}. Finally, the inequality \[\frac{1}{2} \left(1+\sqrt{1-8n+16 m} \right) < 1 + 2\sqrt{m}\] is easy to verify. 
\end{proof}
It is easy to construct graphs, in particular acyclic graphs, satisfying equality in Proposition \ref{prop:max_deg}, namely having $n - \Delta +1$ vertices of degree $1$ and the remaining vertices having degree $2, 3, \ldots, \Delta$. However, to find constructions of such graphs which are also amoebas is much harder. We leave as an open problem to characterize the family of global amoebas that attains this bound (see Problem \ref{prob:max_deg} in Section \ref{sec:problems}).

 \section{Basic problems about amoebas}\label{sec:problems}
 
 In this section, we discuss some problems that arise naturally from the concepts of global and local amoebas and the theory developed in this paper.

One of our main interests is to find more families of local and global amoebas as well as to develop more methods to construct them. Observe that, besides the Fibonacci-amoeba trees given in Section \ref{sec:Fibo-trees}, all constructions of global amoebas provided in Section \ref{sec:constr} yield disconnected graphs. In \cite{CHM_rec}, a way of recursively constructing global amoebas is developed, and it is also shown how it can be used to construct the Fibonacci-amoeba trees. This method is interesting but it yields graphs that have many cut vertices. So it would be nice to find other constructions that give rise to global amoebas with higher connectivity. It would be also interesting to know if there are local or global amoebas with all possible edge numbers. 

\begin{problem}\mbox{}
\begin{enumerate} 
\item[(i)] Find other families of global and/or local amoebas. In particular, find other infinite families of connected global amoebas.
\item[(ii)] Is there a global amoeba on $n$ vertices and $m$ edges for every $m$ with $0 \le m \le \lfloor \frac{n^2}{4} \rfloor$?
\item[(iii)] Is there a local amoeba on $n$ vertices and $m$ edges for every $m$ with $0 \le m \le {n \choose 2}$?
\end{enumerate}
\end{problem}

Of course, the recognition problem and its complexity  should be studied. 
To determine if a graph is a local or a global amoeba, one first has to determine which are its feasible edge-replacements, a problem that involves checking if two graphs are isomorphic. The isomorphism problem in graphs has been intensively studied. The best currently accepted theoretical algorithm is due to Babai and Luks \cite{BaLu83}, which has a running time of $2^{\mathcal{O}(\sqrt{n \log n})}$ for a graph on $n$ vertices. A quasi-polynomial time algorithm was announced by Babai in 2015 \cite{Bab15}, but its proof is not yet fully peer-reviewed, see \cite{Bab17}. However, there are many graphs classes in which the isomorphism problem is polynomial \cite{McK, McKPi}. The difference in checking if a graph $G$ of order $n$ is a local or a global amoeba lies on checking if the group $S_G$ is isomorphic to the symmetric group $S_n$, or if $S_{G^*}$ acts transitively on $[n+1]$, where $G^* = G \cup K_1$ (see Theorem \ref{thm:eq}). Both things can be computed in $O(|S| n)$-time, given that $S$ is a set of generators (see \cite{KaOs, Ser}). In \cite{GGM}, among other results, a public repository containing several programs to detect  both types of amoebas is presented, see \cite{GGM_rep}, where a collection of results for sets of non-isomorphic graphs of order up to $10$, as well as non-isomorphic trees of order up to $22$ are shown.

\begin{problem}
What is the computational complexity of determining if a graph G is a global and/or local amoeba?
\end{problem}

A structural characterization of the graphs that are global but not local amoebas or of those that are local but not global, or of those that are both, that could give clues on how they can be constructed or recognized may be an interesting problem.

\begin{problem}
Provide a structural characterization of the following graph families.
\item[(i)] Global amoebas that are not local amoebas.
\item[(i)] Local amoebas that are not global amoebas.
\item[(i)] Graphs that are both, global and local amoebas.
\end{problem}

However, as the above problem could be challenging in general, it could be more doable if restricted to a particular class of graphs. In this line, we have studied the Fibonacci-trees $T_i$ in Section \ref{sec:Fibo-trees} and we have shown that they are global amoebas. We also have shown at the end of Section \ref{sec:Fibo-trees} that $T_5$ is a local amoeba, too, and, while analogous arguments work for $i \le 4$, it is not clear how to proceed for $i \ge 6$.
\begin{problem}\label{probl:Ti_LA}
Which trees are local/global amoebas? Is the Fibonacci-tree $T_i$ a local amoeba for all $i \ge 1$?
\end{problem}

The graph $H_n$ given in Example \ref{ex:Hn} is shown in Theorem \ref{thm:GA_max_e(G)X(G)w(G)}(i) to have the largest density among the global amoebas of minimum degree $1$. We believe this is the family that characterizes the equality.  We state this as a conjecture. 

\begin{conj}\label{conj:max_e(G)}
If $G$ is a global amoeba of order $n$ and minimum degree~$1$, then $e(G) = \lfloor \frac{n^2}{4} \rfloor$ if and only if $G \cong H_n$. 
\end{conj}

For the bound on the chromatic and the clique numbers given in Theorem \ref{thm:GA_max_e(G)X(G)w(G)}(ii), where $H_n$ is also an example for their sharpness, a characterization of the graphs attaining equality would be interesting as well.

\begin{problem}
Characterize the families of global amoebas $G$ of order $n$ and minimum degree $1$ with $\chi(G) = \lfloor \frac{n}{2}\rfloor +1$ (and, hence, $\omega(G) = \lfloor \frac{n}{2}\rfloor +1$ by Theorem \ref{thm:GA_max_e(G)X(G)w(G)}(iv)).
\end{problem}

The graph $H_n$ is also an example of a global amoeba with the largest possible maximum degree, namely $n-1$. We also have shown in Proposition \ref{prop:max_deg} that the maximum degree of a global amoeba with minimum degree $1$ and with $m$ edges is at most $\frac{1}{2} \left(1 + \sqrt{1- 8n + 16 m}\right)$, and the bound is attained for the star forest $K_{1,2} \cup K_{1,3} \cup \dots \cup K_{1, \Delta}$. However, we do not know about connected global amoebas attaining the bound. In particular, it is intriguing to discover what is the maximum possible degree of a global amoeba tree. We recall at this point that, for the Fibonacci-tree family $T_i$, $i \ge 1$, that we discussed in Section \ref{sec:Fibo-trees}, the growing rate of the maximum degree of $T_i$ is logarithmical with respect to its order, but it could be that there are global amoeba trees where the behavior between maximum degree and order is not that drastic and comes rather closer to $\sqrt{n}$.
\begin{problem}\label{prob:max_deg}
Let $\mathcal{G}_n$ be the family of global amoebas of order $n$ and minimum degree $1$.
\begin{enumerate}
\item[(i)] Characterize the family of all graphs $G \in \mathcal{G}_n$ such that $\Delta(G) = \frac{1}{2} \left(1 + \sqrt{1- 8n + 16 m}\right)$.
\item[(ii)] Determine $f(\mathcal{F}_n) = \max \{ \Delta(F) \;|\; F \in \mathcal{F}_n\}$ for different families $\mathcal{F}_n \subseteq \mathcal{G}_n$, like trees, bipartite graphs, connected graphs, etc. 
\item[(iii)] In particular for the case of the family $\mathcal{T}_n$ of trees on $n$ vertices: is $f(\mathcal{T}_n) = \Theta(\sqrt{n})$?
\end{enumerate}
\end{problem}

\section{Appendix: theoretical setting}\label{sec:theor_set}

The following lemma shows that the application of feasible edge-replacements on any copy
$G_{\rho}$ of a graph $G$ leads to a copy $G_{\sigma\, \rho}$, where $\sigma$ is a permutation associated to the performed edge-replacement. 

\begin{lemma}\label{la:properties_FER}
Let $G$ be a graph provided with a labeling $V(G) \mapsto X$ on its vertices. Then, for any $rs \to kl \in R_G^*$,  $\sigma \in \fer_G(rs \to kl) $ and $\rho \in S_X$,
\[G_{\sigma	\, \rho} = G_{\rho} - e +e',\]
where $e = v_{\rho^{-1}(r)}v_{\rho^{-1}(s)}$ and  $e' = v_{\rho^{-1}(k)}v_{\rho^{-1}(l)}$.
\end{lemma}

\begin{proof}
Since $\sigma \in \fer_G(rs \to kl)$, we have
\[ E(G_{\sigma}) = \left(E(G) \setminus \{v_rv_s\} \right) \cup \{v_kv_l\}  
=  \left\{v_iv_j \;|\; ij \in (L_G \setminus \{rs\}) \cup \{kl\} \right\}. \]
On the other side, it can be seen easily that
\[ E(G_{\sigma}) = \{v_iv_j \;|\; \sigma(i) \sigma(j) \in L_G\} = \{v_{\sigma^{-1}(i)}v_{\sigma^{-1}(j)} \;|\; ij \in L_G\}. \]
Thus, we can infer that 
\[\left\{\{i,j\} \;|\; ij \in (L_G \setminus \{rs\}) \cup \{kl\} \right\} 
= \left\{\{\sigma^{-1}(i),\sigma^{-1}(j)\} \;|\; ij \in L_G \right\}.\]
\noindent
Having this, the following equality chain is straightforward.
\begin{align*}
    E(G_{\sigma\rho}) & = \{v_iv_j \;|\; \sigma\rho(i)\sigma\rho(j) \in L_G\}\\
                        & = \left\{v_{\rho^{-1}(\sigma^{-1}(i))}v_{\rho^{-1}(\sigma^{-1}(j))} \;|\; ij \in L_G \right\}\\
                        & = \left\{v_{\rho^{-1}(i)}v_{\rho^{-1}(j)} \;|\; ij \in (L_G \setminus \{rs\}) \cup \{kl\}  \right\}\\
                         & = \left( \{v_{\rho^{-1}(i)}v_{\rho^{-1}(j)} \;|\; ij \in L_G\}  \,\setminus\, \{v_{\rho^{-1}(r)}v_{\rho^{-1}(s)} \} \right) \cup \{v_{\rho^{-1}(k)}v_{\rho^{-1}(l)}\} \\
                        & = \left( \{v_iv_j \;|\; \rho(i)\rho(j) \in L_G\}  \,\setminus\, \{v_{\rho^{-1}(r)}v_{\rho^{-1}(s)} \} \right) \cup \{v_{\rho^{-1}(k)}v_{\rho^{-1}(l)}\} \\
                        & = E(G_{\rho} - e + e').
\end{align*}
Therefore, $G_{\sigma	\, \rho} = G_{\rho} - e +e'$, as claimed.
\end{proof}

In the next lemma, we establish important facts related to the feasible edge-replacements in non-connected graphs. 

\begin{lemma}\label{la:HUH'}
Let $G = H \cup H'$ be the disjoint union of two graphs $H$ and $H'$. Let $\lambda: V(H) \to X$ and $\lambda ':V(H') \to X'$ be labelings, where $X \cap X' = \emptyset$. Let $\Gamma = \fer(H)$, and $\Gamma' = \fer(H')$. Then $R_H, R_{H'}\subseteq R_G$ and $\Gamma \times \Gamma' \cong  \{ \sigma \cup \tau\;|\; \sigma \in \Gamma, \; \tau \in \Gamma'\} \le {\rm Stab}_{\fer(G)}(X)$.
\end{lemma}
\begin{proof}
It is easy to see that $R_H, R_{H'} \subseteq R_G$. Let $\sigma \in \Gamma$  and $\tau \in \Gamma'$. By definition,  $\sigma = \sigma_q \sigma_{q-1} \cdots \sigma_1$ for certain $\sigma_1, \cdots, \sigma_q \in \mathcal{E}_H$, while $\tau = \tau_{q'} \tau_{q'-1} \cdots \tau_1$ for certain $\tau_1, \cdots, \tau_{q'} \in \mathcal{E}_{H'}$. Without loss of generality, assume that $q \ge q'$. Define $\tau_j = {\rm id}_{\Gamma'}$ for $q'+1 \le j \le q$. For $1 \le j \le q$, let $\rho_j =  \sigma_j \cup \tau_j$.
Since $R_H, R_{H'} \subseteq R_G$ and, for $1 \le j \le q$,
\[G_{\rho_j} = (H \cup H')_{\rho_j} = H_{\sigma_j} \cup H'_{\tau_j},\] 
then $\rho_1, \cdots, \rho_q \in \mathcal{E}_G$. 
Moreover, $\rho  = \rho_q \rho_{q-1} \cdots \rho_1$ and $\rho \in {\rm Stab}_{\Gamma}(X)$.  Since $\Gamma\times \Gamma' \cong  \{\sigma \cup \tau \;|\; \sigma \in \Gamma, \; \tau \in \Gamma'\},$ the latter is a subgroup of ${\rm Stab}_{\fer(G)}(X)$.
\end{proof}

The following lemma deals with edge-replacements that involve, or not, isolated vertices. It is necessary for the proof of Theorem \ref{thm:eq}.

\begin{lemma}\label{la:prop_GUtK_1}
Let $G$ be a graph without isolated vertices, and let $U$ be a set of $t$ vertices disjoint from $V(G)$,  for some $t \ge 1$. Let $G^* = G \cup U \cong G \cup t K_1$. Let $\lambda: V(G) \to X$ and $\lambda: V(G^*) \to X \cup Y$ be labelings such that $\lambda^* \restrict{V(G)} = \lambda$, and let $\Gamma = \fer(G)$ and $\Gamma^* = \fer(G^*)$. We have the following properties:
\begin{enumerate}
\item[(i)] If $e \to e' \in R_G$ and $\sigma \in \Gamma^*(e \to e')$, then $\sigma \in {\rm Stab}_{\Gamma^*}(X)$ and $\sigma\restrict{X} \in \Gamma(e \to e')$. In particular,  $A_{G^*} \leq {\rm Stab}_{\Gamma^*}(X)$ and $A_G = \{\varphi\restrict{X} \;|\; \varphi \in A_{G^*}\}$.
\item[(ii)] If $\sigma \in \mathcal{E}_{G^*}$ is such that with $\sigma(r) = k$ for some $r \in X$ and $k \in Y$, then 
\begin{enumerate}
\item[$\bullet$] either $\sigma =  \varphi \circ (r\; k)$ for some  $\varphi \in \mathcal{E}_{G^*}  \cap {\rm Stab}_{\Gamma^*}(X)$, 
\item[$\bullet$] or $\sigma =  \varphi \circ (r\; k)(s \; l)$ for some  $s \in X\setminus \{r\}$, $l \in Y \setminus\{k\}$ and $\varphi \in A_{G^*}$.
\end{enumerate}
\end{enumerate}
\end{lemma}

\begin{proof}
(i) If $e \to e'= \emptyset \to \emptyset$ and 
$\sigma \in  A_{G^*}$, then $\sigma$ can only permute elements inside $X$ (via an automorphism of $G$) or inside $Y$. Hence, $\sigma \in {\rm Stab}_{\Gamma^*}(X)$, and it follows easily that $A_{G^*} \leq {\rm Stab}_{\Gamma^*}(X)$, and $A_G = \{\varphi\restrict{X} \;|\; \varphi \in A_{G^*}\}$.
Now suppose that $e \to e'\in R_G^*$ and let $e = rs$ and $e' = kl$, where $r,s,k,l \in X$. Since $G$ has no isolates, it follows that  $\sigma \in {\rm Stab}_{\Gamma^*}(X)$. Then, by definition,
\[G^*_{\sigma} = G^* - v_rv_s + v_kv_l = (G - v_rv_s + v_kv_l) \cup X= G_{\sigma\restrict{X}} \cup U.\]
Hence, $\sigma\restrict{X} \in \Gamma(e \to e')$.\\

\noindent
(ii) Since $k \in Y$, we have $\deg_{G^*}(v_k) = 0$. Having that $\sigma \in \mathcal{E}_{G^*}$, $\sigma(r) = k$, and $G$ has no isolates, we conclude that $\deg_{G^*}(v_r) = 1$ in view of Observation \ref{obs:trivialDEG}.  Moreover, there must be some $s \in X$ and some $l \in (X \cup Y)\setminus \{k\}$ such that $\sigma \in \Gamma^*(rs \to kl)$.\\
Suppose first that $l \in X$.
Observe that $(r\;  k)\in \Gamma^*(rs \to ks) \subseteq \mathcal{E}_{G^*}$, and thus $E(G^*_{(r \; k)})=E(G^*- v_rv_s + v_kv_s)$. Therefore,
\begin{align*}
E(G^*_{\sigma}) & =  E(G^* - v_rv_s + v_lv_k)  \\
& = E((G^* - v_rv_s + v_sv_k) -v_sv_k+ v_lv_k) \\
& = E\left(G^*_{(r \; k)}-v_sv_k+ v_lv_k\right).
\end{align*}
But in the copy  $G^*_{(r \; k)}$ the vertex $v_k$ has label $r$, thus $G^*_{\sigma}=G^*_{\varphi  \circ (r \; k)}$, where $\varphi  \in \Gamma^*(sr \to lr)  \subseteq \mathcal{E}_{G^*}$. Finally, since  $r,s,l \in X$, item (i) yields that $\varphi \in {\rm Stab}_{\Gamma^*}(X)$. \\ 
Now suppose that $l \in Y$. In this case we have $\deg_{G^*}(v_k)=\deg_{G^*}(v_l) = 0$, and by Observation \ref{obs:trivialDEG} it follows that $\deg_{G^*}(v_r) =\deg_{G^*}(v_s)  = 1$. Hence,
\[E(G^*_{\sigma}) =  E(G^* - v_rv_s + v_lv_k)  = E\left(G^*_{(r \; k)(s \; l)}\right).\]
It follows now that there is a $\varphi \in A_{G^*}$ such that $\sigma =\varphi \circ (r\; k) (s \; l)$. 
\end{proof}

\section*{Acknowledgements}

We are thankful to the anonymous referees for their much appreciated valuable and constructive suggestions that helped improve this paper.

We also would like to thank BIRS-CMO for hosting the workshop Zero-Sum Ramsey Theory: Graphs, Sequences and More 19w5132, in which the authors of this paper were organizers and participants, and where many fruitful discussions arose that contributed to a better understanding of these topics.

The second and third authors were supported by PAPIIT IG100822.


\begin{thebibliography}{100}


\bibitem{ACKR89} N. Alon, Y. Caro, I. Krasikov, Y. Roditty, Combinatorial reconstruction problems, J. Combin Theory Ser. B 47 (1989), no. 2, 153 -- 161.

\bibitem{Bab15} L. Babai, Graph Isomorphism in Quasipolynomial Time. arXiv preprint, arXiv:1512.03547.

\bibitem{Bab17} L. Babai,  http://people.cs.uchicago.edu/~laci/update.html

\bibitem{BaLu83} L. Babai, E. M. Luks. Canonical labeling of graphs. Proceedings of the Fifteenth Annual ACM Symposium on Theory of Computing (1983), 171--183.

\bibitem{BaMoRa} Y. Bagheri, A. R. Moghaddamfar, F. Ramezani. The number of switching isomorphism classes of signed graphs associated with particular graphs, Discrete Appl. Math. 279 (2020), 25--33. 

\bibitem{BaNiWh} S. Bau, B. van Niekerk, D. White. An intermediate value theorem for the decycling numbers of Toeplitz graphs, Graphs Combin. 31 (2015), no. 6, 2037--2042. 

\bibitem{BeCh} M. Behzad, G. Chartrand, No graph is perfect, Amer. Math. Monthly 74 (1967), 962--963. 

\bibitem{Bon69} J. A. Bondy. Bounds for the Chromatic Number of a Graph, J. Combinatorial Theory 7(1969), 96--98.

\bibitem{CaDuPa} G. C\u{a}linescu, A. Dumitrescu, J. Pach, Reconfigurations in graphs and grids, SIAM J. Discrete Math. 22 (2008), no. 1, 124--138. 

\bibitem{CHM_rec} Y. Caro, A. Hansberg, A. Montejano. Recursive constructions of amoebas, Procedia Computer Science, Volume 195 (2021) 257-265.

\bibitem{CHLZ} Y. Caro, A. Hansberg, J. Lauri, C. Zarb. On zero-sum spanning trees and zero-sum connectivity. Electron. J. Combin. 29 (2022), no. 1, Paper No. 1.9.

\bibitem{CHM3} Y. Caro, A. Hansberg, A. Montejano. Unavoidable chromatic patterns in 2-colorings of the complete graph. J. Graph Theory 97 (2021), no. 1, 123--147. 

\bibitem{CiZi} S. Cichacz, I. A. Zio\l{}o. $2$-swappable graphs, AKCE Int. J. Graphs Comb. 15 (2018), no. 3, 242--250. 

\bibitem{CoFo} D. Conlon, J. Fox, Bounds for graph regularity and removal lemmas, Geometric and Functional Analysis 22 (2012), no. 5,  1191--1256,

\bibitem{Czi05} P. Czimmermann, Representation of the reconstruction conjecture by group actions, Journal of Information, Control and Management Systems 3 (2005), No. 2, 75 -- 80.

\bibitem{Erdos} P. Erd\H{o}s, Some combinatorial, geometric and set theoretic problems in measure theory, in K\"{o}lzow, D. Maharam-Stone, D. (eds.), Measure Theory Oberwolfach 1983, Lecture Notes in Mathematics (1984), vol. 1089, Springer.

\bibitem{FFHHUW} R. Fabila-Monroy, D. Flores-Pe\~{n}aloza, C. Huemer, F. Hurtado, J. Urrutia, D.R. Wood, Token graphs, Gr. Combin. 28(3) (2012), 365--380. 

\bibitem{FrRi} J. Fres\'an-Figueroa, E. Rivera-Campo. On the fixed degree tree graph, J. Inform. Process. (2017) 25, 616--620.

\bibitem{FrHlRo} D. Froncek, A. Hlavacek, S. Rosenberg. Edge reconstruction and the swapping number of a graph, Australas. J. Combin. 58 (2014), 1--15. 

\bibitem{GGM} J. R. González-Martínez, M. E. Gonzalez-Laffitte, A. Montejano. Computational detection of the amoeba graph family and the case of weird edge-replacements (preprint).

\bibitem{GGM_rep} Gonález-Laffite, Marcos, https://github.com/MarcosLaffitte/Amoebas

\bibitem{HaMoPl} F. Harary, R. Mokken, M. Plantholt. Interpolation theorem for diameters of spanning trees, IEEE Trans. Circuits and Systems 30 (1983), no. 7, 429--432. 

\bibitem{HaPl} F. Harary, M. J.  Plantholt, Classification of interpolation theorems for spanning trees and other families of spanning subgraphs, J. Graph Theory. 13 (1989), no.6, 703--712.

\bibitem{HKNS10} S. G. Hartke, H. Kolb, J. Nishikawa, D. Stolee, Automorphism group of a graph and a vertex-deleted subgraph, Electron. J. Combin.  17 (2010), no. 1, Research paper 134, 8 pp.

\bibitem{Hor} D. Horsley. Switching techniques for edge decompositions of graphs, Surveys in combinatorics 2017, 238--271,
London Math. Soc. Lecture Note Ser., 440, Cambridge Univ. Press, Cambridge, 2017. 
 
 
\bibitem{Isa} I.M. Isaacs, Finite Group Theory, in: Graduate Studies in Mathematics, vol. 92, American Mathematical Society, Providence, RI (2008).

\bibitem{IDHPSUU} T. Ito, E. D. Demaine, N. J. A. Harvey, C. H. Papadimitriou, M. Sideri, R. Uehara, Y. Uno, On the complexity of reconfiguration problems,
Theoret. Comput. Sci.  412 (2011), no. 12--14, 1054--1065.

\bibitem{JaPaSc} D. A. Jaume, A. Pastine, V. N. Schv\"ollner. 2-switch: transition and stability on graphs and forests, arXiv preprint, arXiv:2004.11164 , 2020.

\bibitem{KaOs} P. Kaski,  P. R. J. \"Osterg\r{a}rd. Classification algorithms for codes and designs. Algorithms and Computation in Mathematics, 15. Springer-Verlag, Berlin, 2006. 

\bibitem{LaSc} J. Lauri, R. Scapellato. Topics in graph automorphisms and reconstruction. Second edition. London Mathematical Society Lecture Note Series, 432 (2016), xiv+192 pp.

\bibitem{LeTr} J. Lea\~{n}os, A. L. Trujillo-Negrete, The connectivity of token graphs. Graphs Combin. 34 (2018), no. 4, 777--790.

\bibitem{LiLiTo} K. Lih, C. Lin, L.  Tong. On an interpolation property of outerplanar graphs, Discrete Appl. Math. 154 (2006), no. 1, 166--172. 

\bibitem{McK} B. D. McKay, Practical graph isomorphism. Congr. Numer. 30 (1981), 45--87.

\bibitem{McKPi} B. D. McKay, A. Piperno. Practical graph isomorphism, II. J. Symbolic Comput. 60 (2014), 94--112.

\bibitem{MeNoPo} E. Melville, B. Novick, S. Poznanovi\'c, On reconfiguration graphs: an abstraction, Australas. J. Combin. 81 (2021), 339--356.

\bibitem{Pun1} N. Punnim. Interpolation theorems on graph parameters, Southeast Asian Bull. Math. 28 (2004), no. 3, 533--538.

\bibitem{Pun3} N. Punnim. Switchings, realizations, and interpolation theorems for graph parameters, Int. J. Math. Math. Sci. 2005, no. 13, 2095--2117.

\bibitem{RaSc06} A. J. Radcliffe, A. D. Scott, Reconstructing under group actions, Graphs Combin. 22 (2006), 399 -- 419.

\bibitem{Ross} M. S. Ross. 2-swappability and the edge-reconstruction number of regular graphs, arXiv preprint, arXiv:1503.01048, 2015.

\bibitem{Sep19} T. Seppelt, The Graph Isomorphism Problem: Local Certificates for Giant Action, arXiv:1909.10260.

\bibitem{Ser} A. Seress, Permutation group algorithms, Cambridge University Press, 2003.

\bibitem{PoWe} D. J. A. Welsh, M. B. Powell, An Upper Bound for the Chromatic Number of a Graph and Its Application to Timetabling Problems, Comput. J. 10 (1967), 85--86.

\bibitem{West} D. B. West. Introduction to graph theory. Prentice Hall, Inc., Upper Saddle River, NJ (1996), xvi+512 pp.

\bibitem{Zhou2} S. Zhou. Interpolation theorems for graphs, hypergraphs and matroids,
Discrete Math. 185 (1998), no. 1-3, 221--229. 

\end{thebibliography}
\end{document}